\documentclass[10pt, a4]{amsart}

\usepackage{lmodern}
\usepackage[utf8]{inputenc}
\usepackage{amsmath}
\usepackage{graphicx}
\usepackage{amssymb}
\usepackage{esint}
\usepackage[dvipsnames]{xcolor}
\usepackage{tikz}
\usepackage{xxcolor}
\usepackage{floatrow}
\usepackage{color}
\usepackage{amsthm}
\usepackage{epsfig}
\usepackage[english]{babel}
\usepackage{hyperref}
\usepackage[normalem]{ulem}
\usepackage{mathrsfs}
\usepackage{tikz}
\usetikzlibrary{decorations.markings,backgrounds}
\usetikzlibrary{arrows.meta}
\usepackage{pgfplots}
\usepackage{subfigure}
\usepackage{caption}
\usepackage{bbm}

\usepackage{stmaryrd}

\usepackage{float}
\usepackage{pst-plot}

\hypersetup{
	colorlinks,
	linkcolor={red!80!black},
	citecolor={blue!50!black},
	urlcolor={blue!80!black}
}
 \usepackage{diagbox}

\usepackage[a4paper,top=3.5 cm,bottom=3 cm,left=2.5 cm,right=2.5 cm]{geometry}

\theoremstyle{plain}
\newtheorem{lemma}{Lemma}[section]
\newtheorem{proposition}[lemma]{Proposition}
\newtheorem{theorem}[lemma]{Theorem}
\newtheorem{corollary}[lemma]{Corollary}

\theoremstyle{definition}

\theoremstyle{remark}
\newtheorem{remark}[lemma]{Remark}

\numberwithin{equation}{section}

\def\to{\rightarrow}

\usepackage{dsfont}

\newcommand{\nchi}{{\raise.3ex\hbox{$\chi$}}}

\def\XXint#1#2#3{{\setbox0=\hbox{$#1{#2#3}{\int}$ }
		\vcen{\hbox{$#2#3$ }}\kern-.6\wd0}}

\newcommand{\cB}{\mathcal{B}}

\newcommand{\cD}{\mathcal{D}}

\newcommand{\cI}{\mathcal{I}}

\newcommand{\cQ}{\mathcal{Q}}

\DeclareMathOperator{\argmin}{argmin}

\newcommand{\N}{\mathbb{N}}

\newcommand{\R}{\mathbb{R}}
\renewcommand{\S}{\mathbb{S}}

\usepackage{mathtools}

\newcommand{\dd}{\, \mathrm{d}}

\makeatletter
\newcommand{\justified}{%
	\rightskip\z@skip%
	\leftskip\z@skip}
\makeatother
\usepackage{amsfonts}
\usepackage{aurical}

\newcommand\restr[2]{{% we make the whole thing an ordinary symbol
		\left.\kern-\nulldelimiterspace % automatically resize the bar with \right
		#1 % the function
		\vphantom{\big|} % pretend it's a little taller at normal size
		\right|_{#2} % this is the delimiter
}}

\newcommand{\pt}{\partial_t}
\newcommand{\ptt}{\partial_{tt}^2}
\newcommand{\px}{\partial_x }
\newcommand{\pxx}{\partial_{xx}^2}

\renewcommand{\d}{{\rm d}}

\DeclareFontFamily{U}{mathx}{\hyphenchar\font45}
\DeclareFontShape{U}{mathx}{m}{n}{<-> mathx10}{}
\DeclareSymbolFont{mathx}{U}{mathx}{m}{n}
\DeclareMathAccent{\widebar}{0}{mathx}{"73}

\renewcommand{\i}{\ifmmode\mathit{\mathchar"7010 }\else\char"10 \fi}
\renewcommand{\j}{\ifmmode\mathit{\mathchar"7011 }\else\char"11 \fi}

\renewcommand{\le}{\leq} 
\renewcommand{\ge}{\geq}

\def\char{{1\!\mbox{\rm l}}}

\usepackage{dsfont}
\definecolor{orange}{rgb}{1,.549,0}
\definecolor{GreenYellow }{rgb}{ 0.15,   0.69, 0}
\definecolor{Yellowone}{rgb}{ 0, 1., 0} \definecolor{Goldenrod }{rgb}{  0, 0.10, 0.84}
\definecolor{Dandelion }{rgb}{ 0, 0.29, 0.84} 
\definecolor{Apricot }{rgb}{ 0, 0.32, 0.52}
\definecolor{Peach }{rgb}{ 0, 0.50, 0.70} 
\definecolor{GreenYellow}{cmyk}{0.15,0,0.69,0}
\definecolor{RoyalPurple}{cmyk}{0.75,0.90,0,0}
\definecolor{Yellow}{cmyk}{0,0,1,0}
\definecolor{BlueViolet}{cmyk}{0.86,0.91,0,0.04}
\definecolor{Goldenrod}{cmyk}{0,0.10,0.84,0}
\definecolor{Periwinkle}{cmyk}{0.57,0.55,0,0}
\definecolor{Dandelion}{cmyk}{0,0.29,0.84,0}
\definecolor{CadetBlue}{cmyk}{0.62,0.57,0.23,0}
\definecolor{Apricot}{cmyk}{0,0.32,0.52,0}
\definecolor{CornflowerBlue}{cmyk}{0.65,0.13,0,0}
\definecolor{Peach}{cmyk}{0,0.50,0.70,0}
\definecolor{MidnightBlue}{cmyk}{0.98,0.13,0,0.43}
\definecolor{Melon}{cmyk}{0,0.46,0.5,0}
\definecolor{NavyBlue}{cmyk}{0.94,0.54,0,0}
\definecolor{YellowOrange}{cmyk}{0,0.42,1,0}
\definecolor{RoyalBlue}{cmyk}{1,0.50,0,0}
\definecolor{Orange}{cmyk}{0,0.61,0.87,0}
\definecolor{Blue}{cmyk}{1,1,0,0}
\definecolor{BurntOrange}{cmyk}{0,0.51,1,0}
\definecolor{Cerulean}{cmyk}{0.94,0.11,0,0}
\definecolor{Bittersweet}{cmyk}{0,0.75,1,0.24}
\definecolor{Cyan}{cmyk}{1,0,0,0}
\definecolor{RedOrange}{cmyk}{0,0.77,0.87,0}
\definecolor{ProcessBlue}{cmyk}{0.96,0,0,0}
\definecolor{Mahogany}{cmyk}{0,0.85,0.87,0.35}
\definecolor{SkyBlue}{cmyk}{0.62,0,0.12,0}
\definecolor{Maroon}{cmyk}{0,0.87,0.68,0.32}
\definecolor{Turquoise}{cmyk}{0.85,0,0.20,0}
\definecolor{BrickRed}{cmyk}{0,0.89,0.94,0.28}
\definecolor{TealBlue}{cmyk}{0.86,0,0.34,0.02}
\definecolor{Red}{cmyk}{0,1,1,0}
\definecolor{Aquamarine}{cmyk}{0.82,0,0.30,0}
\definecolor{OrangeRed}{cmyk}{0,1,0.50,0}
\definecolor{BlueGreen}{cmyk}{0.85,0,0.33,0}
\definecolor{RubineRed}{cmyk}{0,1,0.13,0}
\definecolor{Emerald}{cmyk}{1,0,0.50,0}
\definecolor{WildStrawberry}{cmyk}{0,0.96,0.39,0}
\definecolor{JungleGreen}{cmyk}{0.99,0,0.52,0}
\definecolor{Salmon}{cmyk}{0,0.53,0.38,0}
\definecolor{SeaGreen}{cmyk}{0.69,0,0.50,0}
\definecolor{CarnationPink}{cmyk}{0,0.63,0,0}
\definecolor{Green}{cmyk}{1,0,1,0}
\definecolor{Magenta}{cmyk}{0,1,0,0}
\definecolor{ForestGreen}{cmyk}{0.91,0,0.88,0.12}
\definecolor{VioletRed}{cmyk}{0,0.81,0,0}
\definecolor{PineGreen}{cmyk}{0.92,0,0.59,0.25}
\definecolor{Rhodamine}{cmyk}{0,0.82,0,0}
\definecolor{LimeGreen}{cmyk}{0.50,0,1,0}
\definecolor{Mulberry}{cmyk}{0.34,0.90,0,0.02}
\definecolor{YellowGreen}{cmyk}{0.44,0,0.74,0}
\definecolor{RedViolet}{cmyk}{0.07,0.90,0,0.34}
\definecolor{SpringGreen}{cmyk}{0.26,0,0.76,0}
\definecolor{Fuchsia}{cmyk}{0.47,0.91,0,0.08}
\definecolor{OliveGreen}{cmyk}{0.64,0,0.95,0.40}
\definecolor{Lavender}{cmyk}{0,0.48,0,0}
\definecolor{RawSienna}{cmyk}{0,0.72,1,0.45}
\definecolor{Thistle}{cmyk}{0.12,0.59,0,0}
\definecolor{Sepia}{cmyk}{0,0.83,1,0.70}
\definecolor{Orchid}{cmyk}{0.32,0.64,0,0}
\definecolor{Brown}{cmyk}{0,0.81,1,0.60}
\definecolor{DarkOrchid}{cmyk}{0.40,0.80,0.20,0}
\definecolor{Tan}{cmyk}{0.14,0.42,0.56,0}
\definecolor{Purple}{cmyk}{0.45,0.86,0,0}
\definecolor{Gray}{cmyk}{0,0,0,0.50}
\definecolor{Plum}{cmyk}{0.50,1,0,0}
\definecolor{Black}{cmyk}{0,0,0,1}
\definecolor{Violet}{cmyk}{0.79,0.88,0,0}
\definecolor{White}{cmyk}{0,0,0,0}
\definecolor{rltred}{rgb}{0.75,0,0}
\definecolor{rltgreen}{rgb}{0,0.5,0}
\definecolor{oneblue}{rgb}{0,0,0.75}
\definecolor{marron}{rgb}{0.64,0.16,0.16}
\definecolor{forestgreen}{rgb}{0.13,0.54,0.13}
\definecolor{purple}{rgb}{0.62,0.12,0.94}
\definecolor{dockerblue}{rgb}{0.11,0.56,0.98}
\definecolor{freeblue}{rgb}{0.25,0.41,0.88}
\definecolor{myblue}{rgb}{0,0.2,0.4}
\definecolor{Melon}{rgb}{ 0.46, 0.50, 0}
\definecolor{Melone}{rgb}{ 0, 0.46, 0.50}

\allowdisplaybreaks
\usepackage[colorinlistoftodos]{todonotes}
\usepackage{url}

\usepackage{graphicx,tikz} 

\title[Decay of 1-d hyperbolic systems]{On the decay of one-dimensional locally and partially dissipated hyperbolic systems}

\author[T. Crin-Barat]{Timothée Crin-Barat}
\address[T. Crin-Barat]{Chair of Computational Mathematics, Fundación Deusto,	Avenida de las Universidades, 24, 48007 Bilbao, Basque Country, Spain.}
\email{timothee.crin-barat@deusto.es}

\author[N. De Nitti]{Nicola De Nitti}
\address[N. De Nitti]{Friedrich-Alexander-Universität Erlangen-Nürnberg, Department of Data Science, Chair for Dynamics, Control and Numerics (Alexander von Humboldt Professorship), Cauerstr. 11, 91058 Erlangen, Germany.}
\email{nicola.de.nitti@fau.de}

\author[E. Zuazua]{Enrique Zuazua}
\address[E. Zuazua]{Friedrich-Alexander-Universit\"at Erlangen-N\"urnberg, Department of Data Science, Chair for Dynamics, Control and Numerics (Alexander von Humboldt Professorship), Cauerstr. 11, 91058 Erlangen, Germany. 
	\newline \indent 
	Chair of Computational Mathematics, Fundación Deusto,	Avenida de las Universidades, 24, 48007 Bilbao, Basque Country, Spain. 
	\newline \indent
	Universidad Autónoma de Madrid, Departamento de Matemáticas, Ciudad Universitaria de Cantoblanco, 28049 Madrid, Spain.}
\email{enrique.zuazua@fau.de}



\begin{document}
\maketitle

\begin{abstract}
We study the time-asymptotic behavior of linear hyperbolic systems under partial dissipation which is localized in suitable subsets of the domain. More precisely, we recover the classical decay rates of partially dissipative systems satisfying the stability condition (SK) with a time-delay depending only on the velocity of each component and the size of the undamped region. To quantify this delay, we assume that the undamped region is a bounded space-interval and that the system without space-restriction on the dissipation satisfies the stability condition (SK). The former assumption ensures that the time spent by the characteristics of the system in the undamped region is finite and the latter that whenever the damping is active the solutions decay. Our approach consists in reformulating the system into $n$ coupled transport equations and showing that the time-decay estimates are delayed by the sum of the times each characteristics spend in the undamped region.
\end{abstract}

\section{Introduction}
\label{sec:intro}

We consider linear hyperbolic systems of the form
\begin{align}\label{eq:LS}
\begin{cases}
\pt U + A \px U = B(x) U, & (x,t) \in \R \times (0,\infty), \\
U(0,x) = U_0(x), & x \in \R,
\end{cases}
\end{align}
where $n\in \N^*$, $U: (0,\infty) \times \R \to \R^n$ is the unknown function such that $U = (u_1,u_2) \in \R^{n_1}\times \R^{n_2}$ and $A,B(x)$ are $n\times n$ real matrices.

We assume that the $x$-dependent matrix $B$ satisfies
\begin{align}\label{eq:damping}
B(x) & > 0, \quad x \in \omega, \qquad  \text{ and } \qquad  B(x) = 0, \quad x \not\in \omega, \\  
B(x) &:=  \begin{pmatrix} 0_{n_1\times n_2} & 0_{n_1\times n_2} \\ 0_{n_2 \times n_1} &  \mathbbm{1}_\omega(x) D \end{pmatrix},
\end{align}
where 
\begin{align}\label{eq:omega}
\omega := \R \setminus B_R(0) = \{x \in \R: \Vert x \Vert \ge R\} \quad \text{ for a fixed $R >0$, }
\end{align}
and
\begin{align}\label{ass:D}
D \in \R^{n_2 \times n_2}, \quad X^T D X > 0, \quad \forall X \in \R^{n_2} \setminus \{0\}.
\end{align}
In other words, the damping term \eqref{eq:damping} acts only on $n_2$ components of the system and is effective only in the region $\omega$. As in \cite{MR1146833}, the choice of $\omega$ as an exterior domain is motivated by a \emph{geometric control condition} (see \cite{MR1178650}): the ray of geometric optics may escape the damping effect and not satisfy any decay properties if the inclusion $\{\Vert x \Vert \ge R\} \subset \omega$ is not satisfied for some $R>0$.

We assume that $A$ is a \emph{strictly hyperbolic matrix} and that $A$ has $n$ real distinct eigenvalues such that
\begin{align}\label{ass:eigR}
\lambda_1<\dots <\lambda_p < 0 < \lambda_{p+1}<\dots < \lambda_n.
\end{align}
In particular this implies that $A$ has no zero eigenvalues
\begin{align}\label{ass:eig0}
	\lambda_i \neq 0, \quad i = 1, \dots, n.
\end{align}
Without this assumption, a ray may never reach the damped region and around it one can always concentrate an undamped solution, a situation we clearly need to avoid to ensure the solutions decay to $0$ when $t\to\infty.$

Analogous assumptions are commonly used in works on the boundary controllability of (systems of) conservation laws (see, e.g., \cite{MR2655971}): when there are zero eigenvalues, i.e. 
\begin{align*}
&\lambda_p < \lambda_q \equiv 0 < \lambda_r, \\
&\nonumber p = 1, \dots, l, \quad q = l+1, \dots, m, \quad r = m+1, \dots, n,
\end{align*}
there are standing waves solutions that may not enter the region of the space where the damping is effective; therefore, to realize exact controllability, suitable boundary control corresponds to non-zero eigenvalues and suitable internal controls correspond to zero eigenvalues.
\medbreak
In the case $\omega=\mathbb{R}$, the existence and the behavior of the solutions of \eqref{eq:LS} is well known.
From the classical theory (see e.g. \cite{Russel}), we know that \eqref{eq:LS} generates a semigroup $S_d(t)$ of bounded operators on $L^2(\mathbb{R};\mathbb{R})$. Therefore, given an initial data $U_0\in L^2$, the system \eqref{eq:LS} has a unique solution $U\in\mathcal{C}(\mathbb{R}_+;L^2(\mathbb(\R))$ such that
$$U(\cdot,t)=S_d(t)U_0. $$

Indeed, applying the Fourier transform (in the space variable) to \eqref{eq:LS} yields
\begin{align*}
\pt \hat U(t,\xi) + i \sum_{j=1}^d A_j \xi_j \hat U(t,\xi) = -B \hat U(t,\xi),
\end{align*}
or, in a condensed form,
\begin{align*}
\pt \hat U(t,\xi) = E(\xi) \hat U(t,\xi),
\end{align*}
where $E(\xi):= -B - iA(\xi)$ and $A(\xi) := \displaystyle\sum_{j=1}^d \xi_j A_j$. 
Solving this first order ODE, we obtain 
\begin{align*}
\hat U(\xi,t) = \exp(E(\xi)t)\hat w_0(\xi).
\end{align*}
Then the $C_0$-semigroup $S_d$ acting on $L^2(\mathbb{R})^n$ can be defined as
$$S_d(t)U=e^{-tE}U=\mathcal{F}^{-1}\left(e^{-tE(\xi)}\widehat U\right)$$
where $-E=-A\partial_xU+BU$ is the associated generator. Its domain contains the Sobolev space $H^1(\mathbb{R})^n$ and with Fourier-Plancherel theorem, the estimate of the semigroup $e^{-tE}$ in $L^2$ is reduced to the analysis of $e^{-tE(\xi)}$ for $\xi\in \R^*$.
\medbreak
Concerning the asymptotic behavior of solutions, in general the semigroup $S_d$ is not dissipative and we have $|||S|||=c$ for some constant $c>0$. 
Indeed, owing to the symmetry of the matrix $A$,  the classical energy method leads to 
  \begin{equation}\label{eq:ZZ}
   \frac 12\frac {\d}{\d t}\|U\|_{L^2}^2 + (BU|U)_{L^2} =0.\end{equation}
 Conditions \eqref{ass:D}   implies that
 there exists $\kappa_0>0$ such that
 \begin{equation}\label{eq:LZ}    (BU|U)_{L^2} \geq \kappa_0\|U_2\|_{L^2}^2.%\quad\hbox{for all } \ Z\in L^2(\R^d;\R^n).
 \end{equation}
 Hence, \eqref{eq:ZZ} yields $L^2$-in-time integrability on the component $w_2$ but not for $w_1$.
 In general, this lack of coercivity greatly reduces our chances of proving decay estimates and one is actually able to construct solution that do not decay at infinity.

To overcome this issue, Shizuta and Kawashima in \cite{MR798756} developed a condition, the well-known stability condition (SK), which ensures the decay of our solution to $0$ when $t\to\infty$:
\begin{align}\label{eq:SK-eig}\tag{SK}
\{\text{eigenvectors of } A(\xi)\} \cap \mathrm{Ker}(B) = \{0\}, \quad \forall\xi \in \R^*. 
\end{align}
In one space dimension, this condition is equivalent to the fact that there are no plane wave solutions to the hyperbolic system propagating to the characteristic directions and therefore ensures the decay of our solutions; in higher dimensions, it is known that the (SK) condition is sufficient to ensure such properties but not necessary. We refer to \cite{MR2754341} for further details

Under the (SK) assumption, in the case $\omega = \R$, one can prove that $\lim_{t\rightarrow\infty}e^{-tE(\xi)}=0$ and more precisely
$$\|e^{-tE(\xi)}\|_{\mathbb{C}^m\rightarrow \mathbb{C}^m}\leq Ce^{-c\frac{|\xi|^2}{1+|\xi|^2}t}. $$
This implies that the behavior of the solution will strongly depend on the frequency regime we are looking at. To this matter, inspired by \cite{MR2349349}, we define the high and low frequencies semigroups $S_d^h$ and $S_d^\ell$ such that
$$U(t)=S_d(t,0)U_0:=S_d^h(t,0)U_0+S_d^\ell(t,0)U_0$$ where the exact formulation of $S_d^h$ and $S_d^\ell$ can be found in \cite[Eq. (3.57) \& (3.82)]{MR2349349}. One interpretation of this formula is that the diffusive part $S_d^\ell$, consists of heat kernels moving along the characteristic directions of the local
relaxed hyperbolic problem and the singular part $S_d^h$ consists of exponentially decaying functions along the characteristic directions of the full system (cf. \cite[p. 1561]{MR2349349}).

We recall the following classical result (see, e.g., \cite[Theorem 1.2]{MR798756}).
\begin{theorem}[SK decay estimate] \label{ThmSK}
Under the condition (SK) and in the case $\omega=\R^d$, the solution of \eqref{eq:LS} verifies
\begin{align}
& \label{eq:decay_SK_h}\Vert U^h(\cdot,t) \Vert_{L^2(\R^d)} \le \|S_d^h(t,0) U_0^h \Vert_{L^2(\R^d)} \le C e^{-\gamma t} \Vert U_0^h \Vert_{L^2(\R^d)}, \\
& \label{eq:decay_SK_l} \Vert U^\ell(\cdot,t) \Vert_{L^\infty(\R^d)} \le \|S_d^\ell(t,0) U_0^\ell \Vert_{L^\infty(\R^d)}\le C t^{-d/2} \Vert U_0^\ell \Vert_{L^1(\R^d)},
\end{align}
where $C,\gamma$ are positive constants depending only on $A$ and $B$, $\widehat{U}^h(\xi,t)=\widehat{U}(\xi,t)\mathds{1}_{\{|\xi|>1\}}$ and $\widehat{U}^\ell(\xi,t)=\widehat{U}(\xi,t)\mathds{1}_{\{|\xi|<1\}}$.
\end{theorem}

Notice that the high frequencies of the solution are exponential damped while the low frequencies behaves as the solutions of the heat equation. Moreover, their computations also allows to deduce the existence of global smooth solutions for nonlinear systems associated to initial data close to a constant equilibrium in any dimension (see \cite{MR2754341,MR2005637,MR2349349,MR2058165,CBD1,CBD2}).

In \cite{MR2754341}, the authors noticed that the condition (SK) is equivalent to the classical \emph{Kalman rank condition} (see \cite{MR889459}) in control theory for all the pairs $(A(\xi), B)$  with $\xi \neq 0$; then they obtained a simpler proof of the decay estimates  \eqref{eq:decay_SK_h}-\eqref{eq:decay_SK_l} by constructing an energy functional with additional low-order terms (motivated by the hypo-coercivity theory of Villani \cite{MR2562709} and by works on the damped wave equation).
More recently, inspired by the construction of such functionals, in \cite{CBD3}, Crin-Barat and Danchin developed a method that allows to study  general quasi-linear partially dissipative hyperbolic systems in a critical regularity framework; in addition, they justified the relaxation process associated to such system by highlighting a purely damped mode in the low-frequency regime which allows to diagonalize the system in this regime.

In the present paper, we investigate system \eqref{eq:LS} in case of a damping acting in a region $\omega$ as in \eqref{eq:omega} instead of the whole space. In this case, the tools developed in the above references are not of much help as they mainly rely on the Fourier transform, which would yield a convolution between $\widehat{\mathbbm{1}_\omega}$ and $\widehat U$ that mixes the frequencies too much to obtain useful information about the dissipative mechanism. 

In the model case of the wave equation, the decay of solutions when the damping term acts only in a region of the domain satisfying a suitable geometric condition has been an active area of investigation in the past decades. In \cite{MR1146833}, Zuazua proved the energy decay for the Klein-Gordon equation with locally distributed dissipation. In \cite{MR1032629}, on the other hand, the decay estimate is obtained for the damped wave equation in a bounded domain. Subsequently, various local energy decay results have been obtained for the linear wave equation in an (unbounded) exterior domain $\Omega \subset \R^d$, which has a localized dissipation being effective only near a part of the boundary (see \cite{MR1855430}). More generally, in \cite{MR361461}, the case of systems with total dissipation (and in a compact domain) is dealt with.  Finally, in \cite{MR3626005,MR3560359}, Léautaud and Lerner studied the decay rate for the energy of solutions of a damped wave equation in a situation where the geometric control condition is violated. 

In a similar spirit, in \cite{CoronNguyen}, the authors study the controllability of general linear hyperbolic system in one space dimension using boundary
controls on one side. Interestingly, a limitation similar to the one we are facing here appears in their computations. Indeed, when dealing with system with three or more components, it appears that their system is not, in general, null-controllable in \textit{optimal time}. And, as we shall see, an analogous phenomenon also occurs in our situation.
In \cite{MR3606411}, the authors deal with 
the controllability from the interior of an hyperbolic system with a reduced number of controls (which parallels the stabilization from $\omega$ using a partial damping). For coupled waves, some similar results are contained in \cite{MR3846289}; we refer also to \cite{MR4086584,MR3550851} for results concerning parabolic or parabolic-hyperbolic systems. 

However, it is not clear how to adapt the above methods to our situation. Here, we take a step-back from these advanced theories and develop a simple and direct method involving only the consideration of the characteristics curves and a semigroup-wise decomposition.

\subsection*{Outline}
\label{ssec:outline} 

In Section \ref{sec:main}, we state our main results. Section \ref{sec:prelim} is dedicated to establishing notations and some preliminary results on hyperbolic systems and characteristics. In Section \ref{sec:strategy-toys}, we outline the strategy and main difficulty of the proof and present the analysis of some toy problems (including the scalar case $n=1$). 
Section \ref{sec:proof-linear} is devoted to the proof of our main result. In Section \ref{sec:3comp}, we give an optimal analysis in the case of a system with $3$ negative eigenvalues.
Finally, in Section \ref{sec:open}, we discuss potential extensions of the main results and several open problems.

\section{Main results}
\label{sec:main}

The main result we establish is the following.
\begin{theorem}[Decay estimates for locally-undamped partially dissipative systems]\label{MainThm}
Assume that the matrix $A$ is symmetric, satisfies \eqref{ass:eigR} and that the couple $(A,B)$ satisfies the (SK) condition.
Let $U$ be the solution of \eqref{eq:LS} associated with the initial data $U_0\in L^1(\R)\cap L^2(\mathbb{R})$.

There exists a constant $C>0$ and a finite time $\bar{\tau}>0$ such that for $t\geq\bar{\tau}$ 
\begin{align*}
& \Vert U^h(\cdot,t) \Vert_{L^2(\R)} \le  Ce^{-\gamma (t-\bar{\tau})} \Vert U_0 \Vert_{L^2(\R)}, \\
&  \Vert U^\ell(\cdot, t) \Vert_{L^\infty(\R)} \le C (t-\bar{\tau})^{-1/2} \Vert U_0 \Vert_{L^1(\R)}
\end{align*}
and $$\displaystyle\bar{\tau}=\max\left(\sum_{i=1}^p\dfrac{2R}{|\lambda_i|},\sum_{i=p+1}^n\dfrac{2R}{|\lambda_i|}\right).$$
\end{theorem}

Next, we establish a theorem describing the behavior of the solution before the time $\bar{\tau}.$
\begin{theorem}[Behavior before $\bar \tau$] \label{MainThm2}
Let the assumptions from Theorem \ref{MainThm} be in force. There exists a time $\tau^*\in[0,\bar{\tau}]$ such that the following holds:
\begin{itemize}
    \item for $t\in[0,\tau^*]$ and $p\in[2,\infty]$, the solution does not decay in time globally in space and we only have
    \begin{align} \label{ThmConserv1}
&\Vert U(t) \Vert_{L^p(\R)} \leq \Vert U_0 \Vert_{L^p(\R)};
\end{align}
\item for $t\in[\tau^*,\bar{\tau}]$, the solution still verifies \eqref{ThmConserv1} but can undergo global-in-space time-decay.
\end{itemize}

Additionally, \begin{align}\label{TauStarTauBar}\tau^*=\bar{\tau}\Leftrightarrow \dfrac{|\lambda_i|}{|\lambda_{i+1}|}=\dfrac{|\lambda_{i+1}|}{|\lambda_{i+2}|} \quad\forall  i\in[1,p-2] \quad \text{or} \quad\forall  i\in[p+1,n-2]),
\end{align}
the choice depending on whether the maximum in the definition of $\bar{\tau}$ is attained for the components associated to the positive or to the negative eigenvalues.
\end{theorem}

\medbreak

\begin{remark} 
A few remarks are in order.
\begin{itemize}
 \item On the time-interval $[\bar{\tau},+\infty]$ the result from Theorem \ref{MainThm} is natural and optimal in the sense that we recover the same decay rates as in the case $\omega=\R$ delayed by the time each characteristics spent in the undamped region $\omega^c$.
 
\item In the exponential decay inequality, it seems that there are too many constants in the sense that if we set $C'=Ce^{\gamma\bar{\tau}}$ then one could have a more compact form. Here, we prefer the two-constant form because we are able to provide an explicit formula for the delay perceived by the time-decay rates of the energy, and $C$ refers to the same constant as in Theorem \ref{ThmSK}.

\item Both our theorems hold without any restriction on the support of the initial data. However, they could be refined when considering a compactly supported initial data. For instance, if one consider initial data supported far from $\omega^c$, then the solution would be decaying during the whole time it takes it to reach $\omega^c$. 

\item The information in Theorem \ref{MainThm2} tells us how the solution behaves for small times. Such a result is useful in the quasilinear context where one needs to know that the undamped region is not too large to prevent the solution from blowing up.
\item The time $\tau^*$ corresponds to the longest time during which some part of the energy of the solution is conserved.

\item The lack of precision on the interval $[\tau^*,\bar{\tau}]$  is due to the fact that 
we are not able to give an explicit formula for $\tau^*$ without involving a very large number of conditions on the eigenvalues of the matrix $A$.

\item In Section \ref{sec:3comp}, we compute the value of $\tau^*$ in the case of a system with three negative eigenvalues ($n=p=3)$ and give a perfect description of the behavior of the solution on the interval $[\tau^*,\bar{\tau}]$ in Theorem \ref{Thm3comp}.

\item The proportionality condition on the eigenvalues of the system is due to fact that the maximum conservation time can only be attained if, for some point, the characteristics continuously enter and exit the undamped region $\omega^c$ without \textit{overlappings} or \textit{gaps}.

\item In Section \ref{sec:open}, we provide a result in the case when $\omega^c$ is a finite union of bounded stripes and discuss other potential extensions.
\end{itemize}
\end{remark}

\section{Preliminaries and notations}
\label{sec:prelim}

Before discussing the strategy of proof, introducing some concepts and notations is necessary. As the matrix $A$ is symmetric with $n$ real distinct eigenvalues, there exists a matrix $P \in O(n,\mathbb{R})$ such that $$P^{-1}AP=D \quad \text{and} \quad D=\operatorname{diag}(\lambda_1,...,\lambda_n).$$
Setting $V=P^{-1}U$, the system \eqref{eq:LS} can be reformulated into 
\begin{align}\label{eq:V}
\begin{cases}
\pt V + D \px V = P^{-1}B(x)PV, & (x,t) \in \R \times (0,\infty), \\
V(0,x) = V_0(x), & x \in \R,
\end{cases}
\end{align}

\begin{remark}
Note that if the matrices $A$ and $B$ commute, as the (SK) condition is satisfied we could diagonalize the matrices simultaneously and end up with decoupled equations as in the totally dissipative case.
\end{remark}
We now notice that a time-decay result for the unknown $V$ immediately implies the same result for the unknown $U$ up to some constant. And it is therefore sufficient to study the unknown $V$ to prove our result.

Decomposing $V=(v_1,\ldots,v_n)$, the system \eqref{eq:V} is equivalent to the following system of coupled transport equations:
\begin{align}\label{eq:vn} 
\begin{cases}
\pt v_1 +\lambda_1 \partial_xv_1 &=\sum_{j=1}^nb_{1,j}v_j\,a(x) \\ &\:\vdots\\
\pt v_n +\lambda_n \partial_xv_n &=\sum_{j=1}^nb_{n,j}v_j\,a(x)
\end{cases}
\end{align}
where the $b_{i,j}$ corresponds to the coefficients of the matrix $P^{-1}BP$.

For instance, the damped wave equation $\ptt u - \pxx u + \pt u = 0$ can be rewritten as 
\begin{align*}
 \begin{cases}
 \pt p - \px p = -\frac{1}{2}\mathbbm{1}_\omega (p+r),\\
 \pt r + \px r = -\frac{1}{2} \mathbbm{1}_\omega (p+r).
 \end{cases}   
\end{align*}

The existence of a unique solution to the linear system \eqref{eq:vn} can be deduce from classical methods to treat transport equations as the boundaries $x=R$ and $x=-R$ are not characteristics. The explicit formulation of the solution via the characteristics method is given in the next section.

\subsection{Characteristics and propagation times} 
\label{ssec:characteristics-def}

For all $1\leq i\leq n$, the characteristic lines \(X_i\) of each equations of system \eqref{eq:vn} passing through the point \(\left(x_{0},t_{0}\right) \in \mathbb{R}\times [0, T]\) are given by 
\begin{align*}
X_i\left(t, x_{0}, t_{0}\right) := \lambda_i(t-t_{0})+x_{0}, \quad t \in[0, T].
\end{align*}

\begin{figure}[H]
{
\begin{tikzpicture}[xscale=0.8,yscale=0.8]
\draw [-latex] (-8,0) -- (10,0) ;
\draw  (5,0)  node [below]  {$x$} ; %x axis

\draw [-latex] (-7,-1) -- (-7,6) ;
\draw  (-7,6)  node [left]  {$t$} ; %t axis

\draw [densely dotted] (-2,0) -- (-2,5) ; %damped region

\draw [densely dotted] (2,0) -- (2,5) ; %damped region

\draw  (-2.5,5)  node [left]  {Damped region} ;

\draw  (1.8,5)  node [left]  {Undamped region} ;

\draw  (6,5)  node [left]  {Damped region} ;

\draw  (-2,0)  node [below]  {$-R$} ;
\draw  (2,0)  node [below]  {$R$} ;

\draw  (4,3)  node [above]  {$(x,t)$} ;

\draw [red, very thick] (-4,0) -- (4,3) ; %characteristic
\draw  (0,2.3)  node [below]  {$X_{1}$}; %Numbercharacteristic

\draw  (-2,0)  node [below]  {$-R$} ;
\draw  (2,0)  node [below]  {$R$} ;

\draw  (4,3)  node [above]  {$(x,t)$} ;

\draw [red, very thick] (-4,0) -- (4,3) ; %characteristic
\draw [red, very thick] (2,0) -- (4,3) ; %characteristic
\draw [red, very thick] (0,0) -- (4,3) ; %characteristic
\draw  (0,2.3)  node [below]  {$X_{1}$}; %Numbercharacteristic
\draw  (0.7,1.5)  node [below]  {$X_{i}$};
%Numbercharacteristic
\draw  (3,1)  node [below]  {$X_{p}$};
%Numbercharacteristic

\draw [blue, very thick] (10,0) -- (4,3);%characteristic
\draw  (8,2)  node [below]  {$X_{n}$};
%Numbercharacteristic
\draw[blue, very thick] (5,0) -- (4,3);%characteristic
\draw [blue, very thick] (8,0) -- (4,3);%characteristic
\draw  (5.2,1.5)  node [below]  {$X_{p+1}$}; %Numbercharacteristic

% \draw  (-7,1.5)  node [red][left]  {$t_{i,ex}$} ;
% \draw [dashed] (-7,1.5) -- (2,1.5);
\end{tikzpicture}  
}
\caption{Characteristics passing through a point $(x,t)\in \R\times\R_+$.} \label{fig:char}
\end{figure}
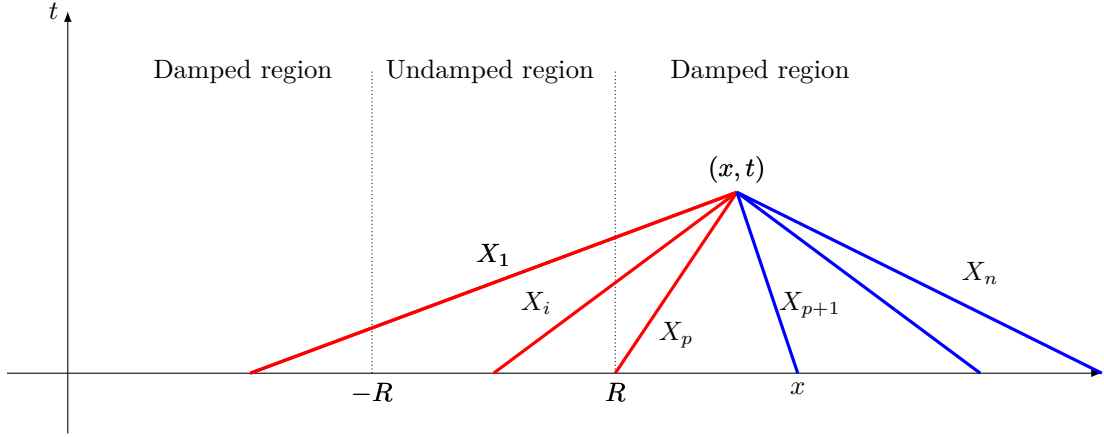
 
Let us already point out two facts that will essential in our study (cf. Figure \eqref{fig:char}):
\begin{enumerate}
    \item Once a characteristic has crossed and exited the undamped region $\omega^c$ it will never cross it again. This allows us to show that the time spend by all the characteristics in $\omega^c$ will be uniformly bounded in $x$ and $t$;
\item Depending on the directions of the characteristics (i.e. the sign of eigenvalues), some characteristics cross $\omega^c$ and some do not.
\end{enumerate}

As the proof of our result revolves very much around time-quantities related to the characteristics and the undamped region $\omega^c$, we to define the following quantities:
\begin{itemize}

\item $t_{i,en}(x_0,t_0)$ (``en'' for ``enter'') the time it takes the characteristic $X_i(\cdot,x_0,t_0)$ to intersect the line $x=-R$ (resp. $x=R$) if $\lambda_i<0$ (resp. $\lambda_i>0$) from the time $t=0$.
If it does not enter $\omega^c$, we set $t_{i,en}(x_0,t_0)=0.$ We have
\begin{align*}
\begin{cases}
t_{i,en}(x_0,t_0)=\max\left(0,t_0-\dfrac{x_0+R}{\lambda_i}\right) \quad \forall\: i\in [1,p],
\\t_{i,en}(x_0,t_0)=\max\left(0,t_0-\dfrac{x_0-R}{\lambda_i}\right) \quad \forall\: i\in [p+1,n].
\end{cases}
\end{align*}
\item $t_{i,ex}(x_0,t_0)$ (``ex'' for ``exit'') the time it takes the characteristic $X_i(\cdot,x_0,t_0)$ to intersect the line $x=R$ (resp. $x=-R$) if $\lambda_i<0$ (resp. $\lambda_i>0$) from $t=0$. If it does not exit $\omega^c$, we set $t_{i,ex}(x_0,t_0)=0.$  We have
\begin{align*}
\begin{cases}
t_{i,ex}(x_0,t_0)=\max\left(0,t_0-\dfrac{x_0-R}{\lambda_i}\right) \quad \forall\: i\in [1,p],
\\t_{i,ex}(x_0,t_0)=\max\left(0,t_0-\dfrac{x_0+R}{\lambda_i}\right) \quad \forall\: i\in [p+1,n].
\end{cases}
\end{align*}
\item $\tau_i(x_0,t_0)$ the time-length during which the characteristic $X_i(\cdot,t_0,x_0)$ is in $\omega^c$ i.e. $$\tau_i(x_0,t_0)=t_{i,ex}(x_0,t_0)-t_{i,en}(x_0,t_0).$$
\item Direct computations lead to an uniform in space and time bound for $\tau_i(x_0,t_0)$:
$$\sup_{x_0\in\R,\, t_0\in[0,T]} \tau_i(x_0,t_0)\leq\dfrac{2R}{|\lambda_i|}.$$

\end{itemize}

\begin{figure}[H]{
\begin{tikzpicture}[xscale=0.8,yscale=0.8]
\draw [-latex] (-8,0) -- (10,0) ;
\draw  (5,0)  node [below]  {$x$} ; %x axis

\draw [-latex] (-7,-1) -- (-7,6) ;
\draw  (-7,6)  node [left]  {$t$} ; %t axis

\draw [densely dotted] (-2,0) -- (-2,5) ; %damped region

\draw [densely dotted] (2,0) -- (2,5) ; %damped region

\draw  (-2.5,5)  node [left]  {Damped region} ;

\draw  (1.8,5)  node [left]  {Undamped region} ;

\draw  (6,5)  node [left]  {Damped region} ;

\draw  (-2,0)  node [below]  {$-R$} ;
\draw  (2,0)  node [below]  {$R$} ;

\draw  (4,3)  node [above]  {$(x,t)$} ;

\draw [red, very thick] (-4,0) -- (4,3) ; %characteristic
\draw  (0,2.3)  node [below]  {$X_{1}$}; %Numbercharacteristic

\draw  (-2,0)  node [below]  {$-R$} ;
\draw  (2,0)  node [below]  {$R$} ;

\draw  (4,3)  node [above]  {$(x,t)$} ;

% \draw [red, very thick] (-4,0) -- (4,3) ; %characteristic
% \draw [red, very thick] (2,0) -- (4,3) ; %characteristic
% \draw [red, very thick] (0,0) -- (4,3) ; %characteristic
% \draw  (0,2.3)  node [below]  {$X_{1}$}; %Numbercharacteristic
% \draw  (0.7,1.5)  node [below]  {$X_{i}$};
% %Numbercharacteristic
% \draw  (3,1)  node [below]  {$X_{p}$};
% %Numbercharacteristic

% \draw [blue, very thick] (10,0) -- (4,3);%characteristic
% \draw  (8,2)  node [below]  {$X_{n}$};
% %Numbercharacteristic
% \draw[blue, very thick] (5,0) -- (4,3);%characteristic
\draw [blue, very thick] (8,0) -- (4,3);%characteristic
\draw  (5.2,1.5)  node [below]  {$X_{p+1}$}; %Numbercharacteristic

\draw  (-7,2.5)  node [red][left]  {$t_{1,ex}(x,t)$} ;
\draw [dashed] (-7,2.27) -- (2,2.27);

% \draw  (-7,1.5)  node [red][left]  {$t_{i,ex}$} ;
% \draw [dashed] (-7,1.5) -- (2,1.5);

\draw  (-7,1)  node [red][left]  {$t_{1,en}(x,t)$} ;
\draw [dashed] (-7,0.8) -- (-2,0.8);

\end{tikzpicture}  
}
\caption{Illustration on the quantities $t_{i,en}(x_0,t_0)$ and $t_{i,ex}(x_0,t_0)$.}
\label{fig:char2}

\end{figure}
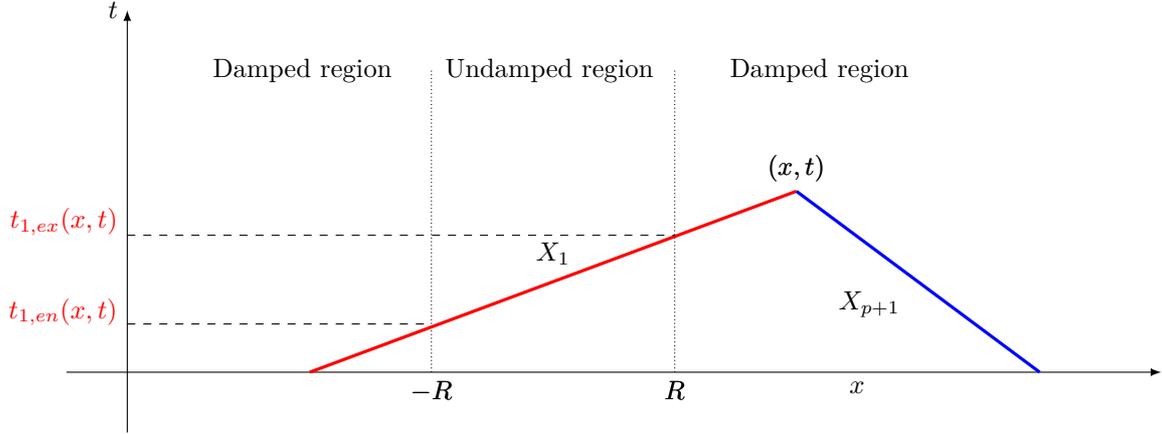

\subsection{Construction of a solution to system \eqref{eq:vn}}
\label{ssec:construction}

Let us now turn to the explicit construction of a solution to \eqref{eq:vn} at the point $(x,t)$ pictured in Figure \ref{fig:char}.

We are going to construct the solution component by component followings their associated characteristics and express it with semigroups.

\begin{remark}[Conservative semigroup]

Inside $\omega^c$, the solution of \eqref{eq:vn} is similar to the solution of \begin{align}\label{eq:NoDampingSyst}
\begin{cases}
\pt V + D \px V =0, & (t,x) \in (0,\infty) \times \R, \\
V(0,x) = V_0(x), & x \in \R,
\end{cases}
\end{align} that does not experience any dissipation.
From classical result, we may define $S_{c}(t)$ ($c$ for conservative) as the semigroup associated with \eqref{eq:NoDampingSyst}.
From a given initial data $V_0\in L^2(\R)$, the solution of \eqref{eq:NoDampingSyst} can be expressed as:
$$V(t)=S_c(t,0)V_0=\begin{pmatrix}v_{1,0}(x-\lambda_1t)\\\vdots\\v_{n,0}(x-\lambda_nt). \end{pmatrix} $$
From standard energy estimates, as the matrix $D$ is diagonal and therefore symmetric, we immediately infer that for $p\in[2,\infty]$,
$$\|V(t,\cdot)\|_{L^p(\R)}=\|S_c(t,0)V_0\|_{L^p(\R)}=\|V_0\|_{L^p(\R)}. $$
\end{remark}
\bigbreak
We shall decompose the semigroup $S_d(t,t'):=(S_{d,1}(t,t'),..,S_{d,n}(t,t'))$ where $S_{d,i}$ is the semigroup associated to the $i$-th equation of \eqref{eq:vn}, and similarly for $S_{c,i}$. Then, looking at the component $v_1$, one can use the method of characteristics in the damped region: going back along the characteristic $X_1(\cdot,x,t)$ until the non-characteristic boundary $x=R$ and deduce that
\begin{align} 
   \nonumber v_1(x,t)&=S_{d,1}(t,t_{1,ex}(x,t))v_1(x,t_{1,ex}(x,t)).
\end{align}
Then, inside $\omega^c$, the conservative semigroup is active  until the non-characteristic boundary $x=-R$ and therefore
\begin{align} 
   \nonumber v_1(x,t_{1,ex}(x,t))
    =S_{c,1}(t_{1,ex}(x,t),t_{1,en}(x,t))v_1(x,t_{1,en}(x,t)). \nonumber
\end{align}
And, finally, the dissipative group is active again:
\begin{align} \nonumber
v_1(x,t_{1,en}(x,t))=
  S_{d,1}(t_{1,en}(x,t),0))v_1(x,0).
\end{align}
This leads to
\begin{align} \label{SdScSd:1} v_1(x,t)=S_{d,1}(t,t_{1,ex}(x,t))S_{c,1}(t_{1,ex}(x,t),t_{1,en}(x,t))S_{d,1}(t_{1,en}(x,t),0)v_{1,0}(x).
\end{align}
Such formulation can be obtained for every point $(x,t)\in\mathbb{R}\times \R_+$ and the definitions of $t_{i,en}$ and $t_{i,ex}$ are made so the formula \eqref{SdScSd:1} is valid at any point.

But one easily remarks that depending on the point $(x,t)$ some characteristics will never cross $\omega^c$ before reaching $(x,t)$ and thus the above formulation can be simplified.
For example, for the component $v_n$ at the point $(x,t)$  pictured in Figure \ref{fig:char}, we have the simpler formulation:
\begin{align} 
   \nonumber v_n(x,t)&=S_{d,n}(t,0)v_1(x,0).
\end{align}

In all generality, following the same arguments as the ones exposed above, we can prove the following proposition.

\begin{proposition}[Representation of solutions to system \eqref{eq:vn}] \label{prop:ConstructSol}
Let $(x,t)\in\R\times \R_+$. For $1 \leq i \leq n$, the solution of \eqref{eq:vn} is given by
\begin{align*}
    v_i(x,t)=S_{d,i}(t,t_{i,ex}(x,t))S_{c,i}(t_{i,ex}(x,t),t_{i,en}(x,t))S_{d,i}(t_{i,en}(x,t),0)v_{i,0}(x).
\end{align*}
More precisely,
\begin{itemize}
    \item if $x \geq R$,
\begin{align*}
 \begin{cases}
  v_i(x,t)=S_{d,i}(t,t_{i,ex}(x,t))S_{c,i}(t_{i,ex}(x,t),t_{i,en}(x,t))S_{d,i}(t_{i,en}(x,t),0)v_{i,0}(x) & \forall\:i\in[1,p],
\\
    v_i(x,t)=S_{d,i}(t,0)v_{i,0}(x)  & \forall\:i\in[p+1,n];
 \end{cases}   
\end{align*}

\item if $x \leq -R$,
\begin{align*}
 \begin{cases}
    v_i(x,t)=S_{d,i}(t,0)v_{i,0}(x)  & \forall\:i\in[1,p],
\\
    v_i(x,t)=S_{d,i}(t,t_{i,ex}(x,t))S_{c,i}(t_{i,ex},t_{i,en}(x,t))S_{d,i}(t_{i,en}(x,t),0)v_{i,0}(x)  & \forall\:i\in[p+1,n];
 \end{cases}   
\end{align*}

\item if $x\in[-R,R]$,
\begin{align*}
 \begin{cases}
    v_i(x,t)=S_{c,i}(t,t_{i,en}(x,t))S_{d,i}(t_{i,en}(x,t),0)v_{i,0}(x)  & \forall\:i\in[1,p],
\\
    v_i(x,t)=S_{c,i}(t,t_{i,en}(x,t))S_{d,i}(t_{i,en}(x,t),0)v_{i,0}(x)  & \forall\:i\in[p+1,n].
 \end{cases}   
\end{align*}

\end{itemize}
\end{proposition}

\begin{remark}[Inhomogeneous transport]
Let us mention that the notation $S_{d,1}(t,t')$ is just a shortcut to the usual formula for the solution of the inhomogeneous transport equation:

$$v_1(x,t)=S_{d,1}(t,t')v(x,t')=v(t',x-\lambda t+\lambda t')+\int_{t'}^tv(s,x-\lambda t+\lambda s)\dd s.$$
Therefore, when the solution has the general form \eqref{SdScSd:1}, we have

\begin{equation}\label{InhoTransp}v_1(x,t)=v_0(x-\lambda t)+\int_{t_{1,ex}(x,t)}^tv(s,x-\lambda t+\lambda s)\dd s+\int_{0}^{t_{1,en}(x,t)}v(s,x-\lambda t+\lambda s)\dd s.
\end{equation}
\end{remark}

Going forward, we shall always assume $V_0\in L^1(\R)\cap L^2(\R)$ and consider the corresponding solution constructed as in Proposition \ref{prop:ConstructSol}.

\section{Strategy of proof and toy problems}
\label{sec:strategy-toys}

\subsection{Key difficulties and strategy of the proof}
\label{ssec:strategy}

The first difficulty encountered when trying to prove time-decay estimates comes from the fact that, on their own, \textit{the semigroups $S_{d,i}$ are not dissipative}: the decay can only be recovered from the coupling between all the equations and the (SK) condition ensures that this coupling generates dissipation for all the components (even for the one that are not directly damped).

To understand the difficulties of our problem, one must have in mind the following facts:
\begin{itemize}
    \item It is only possible to justify dissipation for the solution $V$ if all the semigroups $S_{d,i}$ are  active on a same time-interval i.e. the full semigroup $S_d$ needs to be active (for instance, looking at the effect of $S_{d,1}$ on the first component does not, in general, imply any time-decay properties for the solution $V$ nor for the component $v_1$);
     \item This means that if one of the $S_{c,i}$ is active on a time-interval then the whole solution does not experience any decay on this time-interval;
      \item On their own, the semigroups $S_{d,i}$ are  bounded thanks to the positive semidefiniteness of $B$, and the semigroups $S_{c,i}$ are conservative.
     \end{itemize}

On the other hand, to justify our result, we remark that, the conservative semigroups $S_{c,i}$ are only active on a finite union of finite-time interval. Measuring this union of time-intervals (which depend on the eigenvalues of the matrix $A$) is key in quantifying the delayed time-decay experienced by the solution. As one needs the full semigroup $S_d(t,t')=(S_{d,1}(t,t'),\ldots,S_{d,n}(t,t'))$ to be active on a time-interval to ensure dissipation we have, roughly speaking, that the time spend by each components in the undamped region generates delay for all the other components.

Going forward, we shall restrict the proof to the case $x\geq R$ as the other cases can be dealt with in a similar fashion. From Proposition \ref{prop:ConstructSol}, we have:
\begin{align*}
 \begin{cases}
  v_i(x,t)=S_{d,i}(t,t_{1,ex}(x,t))S_{c,i}(t_{i,ex}(x,t),t_{1,en}(x,t))S_{d,i}(t_{i,en}(x,t),0)v_{i,0}(x) \quad \forall\:i\in[1,p], \\
    v_i(x,t)=S_{d,i}(t,0)v_{i,0}(x)  \quad \forall\:i\in[p+1,n].
 \end{cases}   
\end{align*}

This tells us that, for $1\leq i \leq p$, the dissipative semigroup $S_{d,i}$ is active on the interval $[0,t_{i,en}(x,t)]\cup [t_{i,ex}(x,t),t]$ and the conservative one on $[t_{i,en}(x,t),t_{i,ex}(x,t)]$; on the other hand, for $i\geq p+1$, the dissipative semigroup $S_{d,i}$ is active on the whole interval $[0,t]$.

From this, it follows that at least one component $S_{c,i}$ of $S_c$ is active in the union of intervals
\begin{equation}\label{UnionIntervalDelay} \mathcal{I}(x,t)=\bigcup_{i=1}^p[t_{i,en}(x,t),t_{i,ex}(x,t)]
\end{equation}
and all the components of $S_d$ will be active in its complement. 

The first step of the proof will be to justify rigorously that the delay is directly link to the length of $\cI$. This will follow from arguments that we will construct in the scalar case $n=1$. Notice that as each characteristics cannot spend more than $\dfrac{2R}{|\lambda_i|}$ in $\omega^c$, we have the gross bound
\begin{align}\label{supI}
\sup_{x\geq R,t>0}|\cI(x,t)|\leq \displaystyle\sum_{i=1}^p\dfrac{2R}{|\lambda_i|}.
\end{align}
The use of this inequality will be the last point of the proof of Theorem \ref{MainThm}.
However, in general, the bound in \eqref{supI} is not attained for small times. Indeed, if $t$ is not large enough then there could be overlaps between the intervals of $\mathcal{I}$ and then the upper bound may not be reached for any $x$. From direct computations, we remark that the supremum in \eqref{supI} can be attained whenever there is no overlapping between all the intervals of $\mathcal{I}$.

\subsection{Toy problems}
\label{ssec:toys}

Before tackling the proof of the main theorem, we will present a proof in the scalar case $n=1$ (assuming  $D=\lambda_1<0$) and in the case $n=p=2$. 
\subsubsection{Analysis of the scalar case ($n=1$)}
\label{ssec:scalar}

In this scalar case, one could directly obtain the result from computing explicitly the solution of the system but here we prefer to develop a method that will be applicable to the case of multiple components. Note that the procedure to retrieve decay for the high or low frequencies follows the exact same lines and thus we only focus on the high-frequency one in the sequel.

Let us fix $t>0$; from Proposition \ref{prop:ConstructSol}, we can deduce the following facts:
\begin{itemize}
    \item for $x\geq R$, the dissipative semigroup is active on the time-interval $[0,t_{1,en}(x,t)]\cup [t_{1,ex}(x,t),t]$ and the conservative one on the time-interval $[t_{1,en}(x,t),t_{1,ex}(x,t)]$;
    \item for $-R\leq x \leq R$, the conservative semigroup is active on the time-interval $[t_{en}(x,t),t]$ and the dissipative one on $[0,t_{en}(x,t)]$;
    \item for $x\leq -R$, the dissipative semigroup is active on the time-interval $[0,t]$.
\end{itemize}

 \begin{figure}[H]{
\begin{tikzpicture}[xscale=0.6,yscale=0.6]
\draw [-latex] (-8,0) -- (6,0) ;
\draw  (6,0)  node [below]  {$x$} ; %x axis

\draw [-latex] (-7,-1) -- (-7,6) ;
\draw  (-7,6)  node [left]  {$t$} ; %t axis

\draw [densely dotted] (-2,0) -- (-2,5) ; %damped region

\draw [densely dotted] (2,0) -- (2,5) ; %damped region

\draw  (-4,5)  node [left]  {$\omega$} ;

\draw  (0.7,5)  node [left]  {$\omega^c$} ;

\draw  (4.5,5)  node [left]  {$\omega$} ;

\draw  (-2,0)  node [below]  {$-R$} ;
\draw  (2,0)  node [below]  {$R$} ;

\draw  (0,3)  node [above]  {$(x',t')$} ;

\draw [red, very thick] (-2.8,0) -- (0,3) ; %characteristic
\draw  (-1.58,2.3)  node [below]  {$X_{1}$}; %Numbercharacteristic

\draw  (-2,0)  node [below]  {$-R$} ;
\draw  (2,0)  node [below]  {$R$} ;

\draw  (-3,3)  node [above]  {$(x,t)$} ;

\draw [red, very thick] (-6,0) -- (-3,3) ; %characteristic
\draw  (-4.6,2.3)  node [below]  {$X_{1}$}; %Numbercharacteristic
\draw [red, very thick] (0,0) -- (3,3) ; %characteristic
\draw  (1.4,2.3)  node [below]  {$X_{1}$};
\draw  (3,3)  node [above]  {$(x'',t'')$} ;

\end{tikzpicture}  
}
\caption{Examples of characteristics crossing $\omega$ or $\omega^c$.}
\label{fig:char1}
\end{figure}
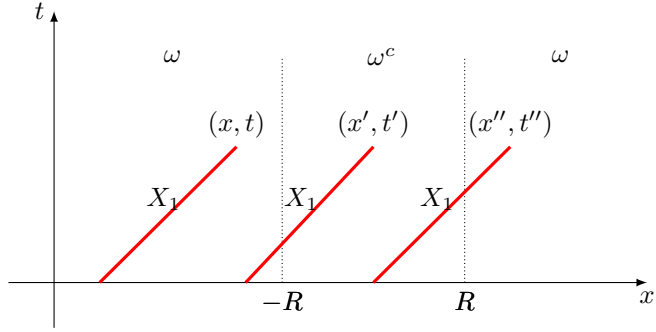

Let us now see how to concretely recover Shizuta-Kawashima's decay estimates for the component $v_1$ with a delay of $\tau_1=\frac{2R}{\lambda_1}$.
To do so, we divide the space region into three parts and bound the time spend by each characteristics in $\omega^c$.
The main difficulty is that the entering and exiting time depend on $x$ and therefore one cannot directly apply the semigroup bound we have from the classical theory. The following lemma solves this issue.

\begin{lemma}For a fixed $t>0$, we have
\begin{align}\label{DecayTex} \|v^h(\cdot,t)\|_{L^2(x\geq R)}\leq e^{-\gamma(t-\bar{\tau})}\|v_0\|_{L^2(\R)}.\end{align}
\end{lemma}
\begin{proof}
For $x\geq R$, we have
$$ v_1(x,t)=S_{d,1}(t,t_{1,ex}(x,t))S_{c,1}(t_{1,ex}(x,t),t_{1,en}(x,t))S_{d,1}(t_{1,en}(x,t),0)v_{1,0}(x).$$
We divide the region $x\geq R$ into three parts: $\mathcal{D}_1\cup \cD_2\cup\cD_3$ where
$$\cD_1:=\{R\leq x\leq t\lambda_1-R\},\quad \cD_2:=\{ t\lambda_1-R\leq x\leq  t\lambda_1+R\},\quad \cD_3:=\{ t\lambda_1+R\leq x\}. $$

\noindent\textbf{Case 1: $x\in\cD_3$.} The proof is straightforward as $t_{1,en}(x,t)=t_{1,ex}(x,t)=0$: the solutions never cross $\omega^c$ and thus decays thanks to the (SK) condition.
We have
\begin{align*}\|v_1(\cdot,t)\|_{L^2(\cD_3)} &\leq e^{-\gamma t}\|v_{1,0}\|_{L^2(\R)}.
 \end{align*}

\noindent\textbf{Case 2: $x\in\cD_1$.} By definition, we have the following pointwise inequalities:
 $$\tau \leq t_{1,ex}(x,t)\leq t, \quad  0\leq t_{1,en}(x,t)\leq t-\tau_1\quad \text{and} \quad t_{1,ex}(x,t)-t_{1,en}(x,t)= \tau_1.$$
 The difficulty is that we cannot use directly these inequalities to prove \eqref{DecayTex} as the entering and exiting times depend on $x$. To solve this we are going to decompose $\cD_3$ into small interval and approximate the solution on each intervals: 
 $$\cD_1=\bigcup_{i=1}^N[a_i,a_{i+1}]\quad \text{s.t. } a_1=R,\;a_n=t\lambda_1-R \text{ and }a_{i+1}-a_i\leq \dfrac{1}{N} \text{ for } N\in\N^*.$$

 On each intervals $[a_i,a_{i+1}]$, using \eqref{InhoTransp}, we have
 \begin{align*}\int_{a_i}^{a_i+1}|v_1(x,t)|^2\dd x&=\int_{a_i}^{a_i+1}\left|v_0(x-\lambda t)+\int_{t_{1,ex}(x,t)}^tv_1(s,x-\lambda t+\lambda s)\dd s+\int_{0}^{t_{1,en}(x,t)}v_1(s,x-\lambda t+\lambda s)\dd s\right|^2\dd x
 \\&= \int_{a_i}^{a_i+1}\left|v_0(x-\lambda t)+\int_{0}^tv_1(s,x-\lambda t+\lambda s)\dd s-\int_{t_{1,en}(x,t)}^{t_{1,ex}(x,t)}v_1(s,x-\lambda t+\lambda s)\dd s\right|^2\dd x.
 \end{align*}
 From here, we define the quantity
 \begin{align*}
 \mathcal{Q}=\int_{0}^tv_1(s,x-\lambda t+\lambda s)\dd s-\int_{t_{1,en}(x,t)}^{t_{1,ex}(x,t)}v_1(s,x-\lambda t+\lambda s)\dd s
 \end{align*}
 and we assume that $u_0(x-\lambda_1t)+\cQ$ is positive; indeed, the case where it is negative can be treated in a similar manner by reversing all the inequalities below.
 
 Let us have a closer look at the second term of $\cQ$. Defining the positive and negative part of the quantity inside the integral by 
 $v_1^+=\max(0,v_1(s,x-\lambda t+\lambda s))$ and $v_1^-=\max(0,-v_1(s,x-\lambda t+\lambda s))$,  we have
  \begin{align*}
 \int_{a_i}^{a_i+1}\int_{t_{1,en}(x,t)}^{t_{1,ex}(x,t)}v_1(s,x-\lambda t+\lambda s)\dd s \dd x &= \int_{a_i}^{a_i+1}\int_{t_{1,en}(x,t)}^{t_{1,ex}(x,t)}v_1^+-v_1^-\dd s \dd x. \end{align*}
Then, thanks to the inequality
$$ t_{1,en}(a_i)\leq t_{1,en}(x,t)\leq t_{1,en}(a_{i+1})\leq t_{1,ex}(a_i)\leq t_{1,ex}(x,t)\leq t_{1,ex}(a_{i+1}),$$
where we omit to write the $t$-dependence inside $t_{1,en}$ for the sake of brevity,  we have $$[t_{1,en}(a_{i+1}),t_{1,ex}(a_{i})]\subset [t_{1,en}((x,t)),t_{1,ex}(x,t)] \subset [t_{1,en}(a_{i}),t_{1,ex}(a_{i+1})],$$ which allows us to bound the integral as follows (see Figure \ref{fig:aibounds}):
\begin{align*}
\int_{a_i}^{a_i+1}\int_{t_{1,en}(a_{i+1})}^{t_{1,ex}(a_i)}v_1^+\leq\int_{a_i}^{a_i+1}\int_{t_{1,en}(x,t)}^{t_{1,ex}(x,t)}v_1^+\leq \int_{a_i}^{a_i+1}\int_{t_{1,en}(a_{i+1})}^{t_{1,ex}(a_i)}v_1^+
 \end{align*}
 and
 \begin{align*}
-\int_{a_i}^{a_i+1}\int_{t_{1,en}(a_i)}^{t_{1,ex}(a_{i+1})}v_1^-\leq-\int_{a_i}^{a_i+1}\int_{t_{1,en}(x,t)}^{t_{1,ex}(x,t)}v_1^-\leq -\int_{a_i}^{a_i+1}\int_{t_{1,en}(a_i)}^{t_{1,ex}(a_{i+1})}v_1^-.
 \end{align*}
 
 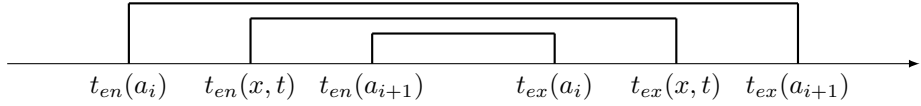
\begin{figure}[H]{
\begin{tikzpicture}[xscale=0.8,yscale=0.8]
\draw [-latex] (-7,0) -- (8,0) ;

\draw  (-5,0)  node [below]  {$t_{en}(a_i)$} ;
\draw  (-3,0)  node [below]  {$t_{en}(x,t)$} ;
\draw  (-1,0)  node [below]  {$t_{en}(a_{i+1})$} ;

\draw  (2,0)  node [below]  {$t_{ex}(a_i)$} ;
\draw  (4,0)  node [below]  {$t_{ex}(x,t)$} ;
\draw  (6,0)  node [below]  {$t_{ex}(a_{i+1})$} ;

\draw[thick] (-1,0) -- (-1,0.5) ;
\draw[thick] (2,0) -- (2,0.5) ;
\draw[thick] (-1,0.5) -- (2,0.5) ;

\draw[thick] (-3,0) -- (-3,0.75) ;
\draw[thick] (4,0) -- (4,0.75) ;
\draw[thick] (-3,0.75) -- (4,0.75) ;

\draw[thick] (-5,0) -- (-5,1) ;
\draw[thick] (6,0) -- (6,1) ;
\draw[thick] (-5,1) -- (6,1) ;

% \draw  (-7,1.5)  node [red][left]  {$t_{i,ex}$} ;
% \draw [dashed] (-7,1.5) -- (2,1.5);

% \draw  (-7,1)  node [red][left]  {$t_{1,en}$} ;
% \draw [dashed] (-7,0.8) -- (-2,0.8);
\end{tikzpicture}  
}\caption{Decomposition of the time-interval used in the proof of Proposition \ref{DecayTex}.}
\label{fig:aibounds}
\end{figure}

 Therefore, we deduce that, on $[a_i,a_{i+1}]$,
 \begin{align*}
 \int_0^tv_1^+-\int_{t_{1,en}(a_{i+1})}^{t_{1,ex}(a_{i})}v_1^+-\int_0^tv_1^-+\int_{t_{1,en}(a_i)}^{t_{1,ex}(a_{i+1})}v_1^-\leq \cQ;
 \end{align*}
 thus
 \begin{align*}
 \int_{t_{1,ex}(a_{i})}^tv_1^++\int_{0}^{t_{1,en}(a_{i+1})}v_1^+-\int_{t_{1,ex}(a_{i+1})}^tv_1^-+\int_{0}^{t_{1,en}(a_i)}v_1^-\leq \cQ,
 \end{align*}
 which can be rewritten as
  \begin{align*}
 \int_{t_{1,ex}(a_{i+1})}^tv_1+\int_{t_{1,ex}(a_{i})}^{t_{1,ex}(a_{i+1})}v_1^++\int_{0}^{t_{1,en}(a_{i})}v_1+\int_{t_{1,en}(a_{i})}^{t_{1,en}(a_{i})}v_1^-\leq \cQ.
 \end{align*}
 
 Similar arguments lead to the upper bound:
  \begin{align*}
 \cQ \leq \int_{t_{1,ex}(a_{i})}^tv_1+\int_{t_{1,ex}(a_{i+1})}^{t_{1,ex}(a_{i})}v_1^++\int_{0}^{t_{1,en}(a_{i+1})}v_1+\int_{t_{1,en}(a_{i})}^{t_{1,en}(a_{i+1})}v_1^-.
 \end{align*}
 Therefore, gathering the above estimates we obtain
\begin{align*}\int_{a_i}^{a_i+1}|v_1(x,t)|^2dx&\leq \int_{a_i}^{a_i+1}\left|v_{1,0}(x-\lambda t)+\int_{t_{1,ex}(a_{i})}^tv_1+\int_{t_{1,ex}(a_{i+1})}^{t_{1,ex}(a_{i})}v_1^++\int_{0}^{t_{1,en}(a_{i+1})}v_1+\int_{t_{1,en}(a_{i})}^{t_{1,en}(a_{i+1})}v_1^-\right|^2
\\&\leq \int_{a_i}^{a_i+1}\left|S_d(t,t_{ex}(a_i))S_c(t_{ex}(a_i),t_{en}(a_{i+1})))S_d(t_{en}(a_{i+1}),0)v_{1,0}(x)+R_{i}\right|^2
 \end{align*}
 where $\displaystyle R_i=\int_{t_{1,ex}(a_{i+1})}^{t_{1,ex}(a_{i})}v_1^++\int_{t_{1,en}(a_{i})}^{t_{1,en}(a_{i+1})}v_1^-$. 
 
 We are now in position to use the dissipative properties of the semigroup $S_{d,1}$ as the entering and exiting time do not depend on $x$ anymore. Bounding the right-hand side integral by the integral on $\R$ and using the properties of the semigroups $S_{d,1}$ and $S_{c,1}$, we get
 \begin{align*}\int_{a_i}^{a_{i+1}}|v_1(x,t)|^2\dd x&\leq e^{-\gamma(t-t_{ex}(a_i))}e^{-\gamma t_{en}(a_{i+1})}\|v_{1,0}\|_{L^2([a_i,a_{i+1}])}+|R_i|^2.
 \end{align*}
 Then, as by definition $t_{1,ex}(a_i)-t_{1,en}(a_{i+1})\leq \tau+\dfrac{1}{N}$ and $t_{en}(a_{i+1}))-t_{en}(a_i)\leq \dfrac{C}{N}$ and similarly for $t_{ex}$, by summing on $i$ and taking the limit as $N\to\infty$, we have
 \begin{align*}\int_{\cD_1}|v_1(x,t)|^2\dd x&\leq e^{-\gamma(t-\tau)}\|v_{1,0}\|_{L^2(\R)}
 \end{align*}
 and therefore 
  \begin{align*}\|v_1(x,t)\|_{L^2(\cD_1)}&\leq e^{-\gamma(t-\tau)}\|v_{1,0}\|_{L^2(\R)}.
 \end{align*}
 
\noindent \textbf{Case 3: $x\in\cD_2$.} We have the following pointwise inequalities:
$$0 \leq t_{1,ex}(x,t)\leq \tau, \quad \text{and} \quad t_{1,en}(x,t)=0.$$
Therefore, with a similar splitting as before, we conclude that
  \begin{align*}\int_{\lambda_1t-R}^{\lambda_1t+R}|v_1^h(x,t)|^2\dd x&\leq e^{-\gamma(t-\tau)}\|v_{1,0}\|_{L^2(\R)}.
 \end{align*}
\end{proof}

\begin{remark}[Decomposition of the initial data]\label{rk:v0-decomp}
We remark that a similar way of proceeding could be to decompose the initial data into $(v_{1,0}^i)_{i=1,\cdots,N}$ such that each $v^i_{1,0}$ is supported in $[a_i-\lambda t,a_{i+1}-\lambda t]$. Then, since the system is linear, the superposition principle applies and we can study each $v_{1}^i$ separately.
\end{remark}

\subsubsection{Analysis of the case $n=p=2$} 
\label{ssec:toy2}

Let us also study a second toy problem, the particular case $n=p=2$, from which we will be able to deduce the general result for $n$ components.

In this case, the solutions of the two transport equations are given by 
\begin{align*}&v_1(x,t)=v_{1,0}(x-\lambda_1 t)+\int_{0}^t\sum_{i=1}^2b_{1,i}v_i(s,x-\lambda_1 t+\lambda_1 s)\dd s-\int_{t_{1,en}(x,t)}^{t_{1,ex}(x,t)}\sum_{i=1}^2b_{1,i}v_i(s,x-\lambda_1 t+\lambda_1 s)\dd s,
\\&v_2(x,t)=v_{2,0}(x-\lambda_2 t)+\int_{0}^t\sum_{i=1}^2b_{2,i}v_i(s,x-\lambda_2 t+\lambda_2 s)\dd s-\int_{t_{2,en}(x,t)}^{t_{2,ex}(x,t)}\sum_{i=1}^2b_{2,i}v_i(s,x-\lambda_2 t+\lambda_2 s)\dd s.
\end{align*}
Denoting $\cB_1:=\sum_{i=1}^2b_{1,i}v_i(s,x-\lambda_1 t+\lambda_1 s)$ and $\cB_2:=\sum_{i=1}^2b_{2,i}v_i(s,x-\lambda_2 t+\lambda_2 s)$, we have
\begin{align} \label{VB1}
    |V(x,t)|^2=\left|\begin{pmatrix}
    v_{1,0}(x-\lambda_1 t)\\v_{2,0}(x-\lambda_2 t)\end{pmatrix}+
    \begin{pmatrix}
    \int_0^t \cB_1 -\int_{t_{1,en}(x,t)}^{t_{1,ex}(x,t)} \cB_1 \\ 
    \int_0^t \cB_2 -\int_{t_{2,en}(x,t)}^{t_{2,ex}(x,t)} \cB_2
    \end{pmatrix}\right|^2.
\end{align}
Let us assume that the quantities in the two rows of \eqref{VB1} are positive, the other three scenarios being treatable in a similar way as we always have upper and lower bounds. Then, defining the sequence $(a_i)_{i\in\{1,\ldots,N\}}$ as in Section \ref{ssec:scalar} on the space-interval $[R,t\lambda_1+R]$ and proceeding as in the scalar case $n=1$, one obtains

\begin{align}
    \int_{a_i}^{a_{i+1}}|V(x,t)|^2=\left|\begin{pmatrix}
   S_{d,1}(t,t_{1,ex}(a_i))S_c(t_{1,ex}(a_i),t_{1,en}(a_{i+1})))S_d(t_{1,en}(a_{i+1}),0)v_{1,0}(x)+R_{1,i}\\S_{d,2}(t,t_{2,ex}(a_i))S_c(t_{2,ex}(a_i),t_{2,en}(a_{i+1})))S_d(t_{2,en}(a_{i+1}),0)v_{1,0}(x)+R_{2,i}\end{pmatrix}\right|^2
\end{align}
where for $j=1,2$
$$\displaystyle R_{j,i}=\int_{t_{j,ex}(a_{i+1})}^{t_{j,ex}(a_{i})}\cB_j^++\int_{t_{j,en}(a_{i})}^{t_{j,en}(a_{i+1})}\cB_j^-.$$

Note that, for $x \geq t\lambda_1+R$, we have $t_{1,en}=t_{2,en}=t_{1,ex}=t_{2,ex}=0$ and therefore we do not need to decompose the space-interval farther.
Then, since it is only possible to recover dissipation when $S_{d,1}$ and $S_{d,2}$ are active on a same time-interval, we obtain
\begin{align*}\int_{a_i}^{a_{i+1}}|V^h(x,t)|^2dx&\leq e^{-\gamma(t-|\cI_i|)}\|V_0\|_{L^2(\R)}+\sum_{j=1}^2|R_{j,i}|^2
 \end{align*}
 where $|\cI_i|=t_{1,ex}(a_i)-t_{1,en}(a_{i+1})+t_{2,ex}(a_{i})-t_{2,en}(a_{i+1})$.
 Summing on $i$ and taking the limit as $N\to\infty$, we obtain
 \begin{align}\label{FinalVn2}
     \|V(t)^h\|_{L^2(x\geq R)}\leq e^{-\gamma(t-\sup_{x\geq R}|\cI(x,t)|)}\|V_0\|_{L^2(\R)}
 \end{align}
 where $\cI(x,t)=[t_{1,en}(x,t),t_{1,ex}(x,t)]\cup[t_{2,en}(x,t),t_{2,ex}(x,t)]$.
 
Similar estimates hold true for any number of components.

\section{Proofs of the main theorems}
\label{sec:proof-linear}

\subsection{Proof of Theorem \ref{MainThm}}
\label{ssec:thm1}
In the general setting, the previous analysis leads to the following corollary.
\begin{corollary}\label{CorI}
Let $V$ be the solution of \eqref{eq:V} associated with the initial data $V_0\in L^1(\mathbb{R})\cap L^2(\mathbb{R})$.

If $x\geq R$, at least one component of the conservative semigroup $S_c$ is active in
\begin{equation} \label{I+} \mathcal{I}(x,t)=\bigcup_{i=1}^p[t_{i,en}(x,t),t_{i,ex}(x,t)].
\end{equation}

If $x\leq -R$, at least one component of the conservative semigroup $S_c$ is active in
\begin{equation} \mathcal{I}(x,t)=\bigcup_{i=p+1}^{n}[t_{i,en}(x,t),t_{i,ex}(x,t)].
\end{equation}
And, if $x\in [-R,R]$, at least one component of the conservative semigroup $S_c$ is active in
\begin{equation} \mathcal{I}(x,t)=\bigcup_{i=1}^{n}[t_{i,en}(x,t),t].
\end{equation}

Moreover,
\begin{equation} \label{BoundI}
\sup_{x\in\mathbb{R},t>0}|\mathcal{I}(x,t)|\leq \bar{\tau}=\max\left(\sum_{i=1}^p \frac{2R}{\lambda_i},\sum_{i=p+1}^n \frac{2R}{\lambda_i}\right).
\end{equation}
\end{corollary}

With this, we are ready for the proof of Theorem \ref{MainThm}.

\begin{proof}[Proof of Theorem \ref{MainThm}]
Using the decomposition from Proposition \ref{prop:ConstructSol}, Corollary \ref{CorI}, and performing similar computations as in the case $n=2$, we can infer that
 \begin{align}\label{FinalVn}
     \|V^h(t)\|_{L^2(\R)}\leq e^{-\gamma(t-\sup_{x\in\R}|\cI(x,t)|)}\,\|V_0\|_{L^2(\R)}.
 \end{align}

Thus, using the bound \eqref{BoundI}, we have that for $t\geq \bar{\tau}$, 
\begin{align*} \|V^h(t)\|_{L^2(\R)}\leq e^{-\gamma (t-\bar{\tau})}\|V_0\|_{L^2(\R)}.
\end{align*} 
\end{proof}

\subsection{Proof of Theorem \ref{MainThm2}}
\label{ssec:thm2}

Let us study the behavior of the solution on the time-interval $[0,\bar{\tau}]$, which appears to highly depend on algebraic relations between the eigenvalues. First, we show that until a certain time the solution does not experience any decay.
\begin{lemma}\label{LemmaCons}
There exists a time $\tau^*$ such that for $t\in[0,\tau^*]$,
    \begin{align*}
&\Vert V(t, \cdot) \Vert_{L^2(\R)} \leq \Vert V_0 \Vert_{L^2(\R)}. \\
\end{align*}
\end{lemma}
\begin{proof}
Let us again focus on the case $x\geq R$.
Proving this lemma amounts in saying that there exists a time $\tau^*$ such that the conservative semigroups $S_{c,i}$ are continuously active from the time $t=0$ to the time $\tau^*.$  Define $\cD$ the characteristic $X_p$ crossing $\left(R,\tfrac{2R}{\lambda_p}\right)$. The points on $\cD$ are the points for which the conservative semigroup $S_{c,p}$ has been active on the whole time-interval $[0,\tau_p]$. Thus, one could chose $\tau_p$ as $\tau^*$ and the lemma is proved.
But one can actually increase this value without additional conditions. Indeed, it is possible to find a point $(x_{p-1},t_{p-1})\in\mathcal{D}$ such that $t_{ex,p}(x_{p-1},t_{p-1})=t_{en,p-1}(x_{p-1},t_{p-1})$; namely, $$x_{p-1}=\dfrac{R(|\lambda_{p}|+|\lambda_{p-1}|)}{|\lambda_{p-1}|-|\lambda_{p}|} \quad \text{and} \quad t_{p-1}=2R\dfrac{|\lambda_p||\lambda_{p-1}|}{|\lambda_{p-1}|-|\lambda_{p}|}.$$
This point is the only point for which the characteristics $X_p$ and $X_{p-1}$ are inside $\omega^c$ in a continuous manner and without overlapping. Therefore, the value of $\tau^*$ can be updated to, at least, $\tau_p+\tau_{p-1}$.
\end{proof}

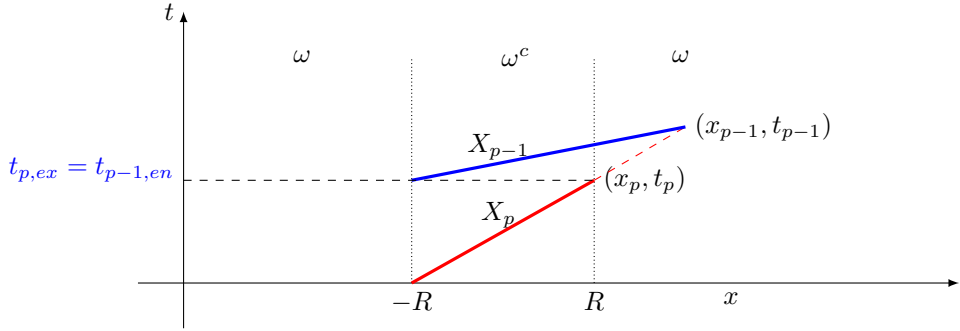
\begin{figure}[H]{
\begin{tikzpicture}[xscale=0.6,yscale=0.6]
\draw [-latex] (-8,0) -- (10,0) ;
\draw  (5,0)  node [below]  {$x$} ; %x axis

\draw [-latex] (-7,-1) -- (-7,6) ;
\draw  (-7,6)  node [left]  {$t$} ; %t axis

\draw [densely dotted] (-2,0) -- (-2,5) ; %damped region

\draw [densely dotted] (2,0) -- (2,5) ; %damped region

\draw  (-4,5)  node [left]  {$\omega$} ;

\draw  (0.8,5)  node [left]  {$\omega^c$} ;

\draw  (4.3,5)  node [left]  {$\omega$} ;

\draw  (-2,0)  node [below]  {$-R$} ;
\draw  (2,0)  node [below]  {$R$} ;

\draw  (2,2.27)  node [right]  {$(x_p,t_p)$} ;

\draw [red, very thick] (-2,0) -- (2,2.27) ; %characteristic
\draw  (-0.1,2)  node [below]  {$X_{p}$}; %Numbercharacteristic
\draw [red, dashed] (2,2.27) -- (4,3.45);
\draw  (4,3.45)  node [right]  {$(x_{p-1},t_{p-1})$} ;

\draw [blue, very thick] (-2,2.27) -- (4,3.45) ; %characteristic
\draw  (-0.1,3.5)  node [below]  {$X_{{p-1}}$}; %Numbercharacteristic

\draw [dashed] (-7,2.27) -- (2,2.27);
 \draw  (-7,2.5)  node [blue][left]  {$t_{p,ex}=t_{{p-1},en}$} ;

\end{tikzpicture}  
}
\caption{Construction of the point $(x_{p-1},t_{p-1})$.}
\label{fig:proofLemmaStar}
\end{figure}
\begin{remark} A similar construction could be done for any couple of characteristics $(X_q,X_k)$ but the largest value is obtained for the couple $(X_p,X_{p-1})$ because these are the slowest characteristics i.e. $\tau_p>\tau_{p-1}>\ldots>\tau_1$.
\end{remark}

One could think of continuing the procedure described above and looking for a point $(x_{p-2},t_{p-2})$ such that $t_{ex,p-1}(x_{p-2},t_{p-2})=t_{en,p-2}(x_{p-2},t_{p-2}).$ But, except if some specific conditions are satisfied by the eigenvalues, one cannot be sure if such a point creates some time-interval where the dissipation is active or if the characteristics overlap in $\omega^c$ and therefore the maximum delay is not attained.

Due to this phenomenon, that we are not able to obtain the largest $\tau^*$ possible satisfying such properties in the general case. This is related to the issue encountered by the authors in \cite{CoronNguyen} where they show counterexamples to the time-optimal null-controllability when considering more than $2$ components. The computation of the constant $\tau^*$ in the case of $3$ components associated to negative eigenvalues is the purpose of Section \ref{sec:3comp}.

The following lemma concludes the proof of Theorem \ref{MainThm2}.
\begin{lemma}\label{lm:th2-2}
The following equivalence holds $$\tau^*=\bar{\tau} \Leftrightarrow\dfrac{|\lambda_i|}{|\lambda_{i+1}|}=\dfrac{|\lambda_{i+1}|}{|\lambda_{i+2}|} \quad\forall  i\in[1,p-2] \text{ or } \forall i\in[p+1,n-2]. $$
\end{lemma}

\begin{proof}
The left-hand side condition says that the energy of the solution is conserved for the maximum time possible in a continuous manner from the time $t=0$ to the time $\bar{\tau}$.
It is clear that this can only happen for the points that are on the line $\cD$, otherwise the time spent inside $\omega^c$ by the $p$-th characteristic on the interval $[0,\tau_p]$ would be strictly smaller than $\tau_p$. Moreover, the point we should look at on $\cD$ is $(x_{p-1},t_{p-1})$ as it is the only such that $t_{ex,p}(x_{p-1},t_{p-1})=t_{en,p-1}(x_{p-1},t_{p-1}).$
Therefore, to conclude the proof of the lemma, one needs to verify under which conditions the point $(x_{p-1},t_{p-1})$ constructed previously satisfies $t_{i,ex}(x_{p-1},t_{p-1})=t_{i-1,en}(x_{p-1},t_{p-1})$ for all $i\in[2,p]$. This can be easily solved explicitly:
 \begin{align*} &t_{i,ex}(x_{p-1},t_{p-1})=t_{i-1,en}(x_{p-1},t_{p-1}) \quad \forall \;i\in[2,p]\\&\Leftrightarrow t_{p-1}-\dfrac{x_{p-1}-R}{|\lambda_{i}|}=t_{p-1}-\dfrac{x_{p-1}+R}{|\lambda_{i-1}|} \quad \forall \;i\in[2,p]\\&\Leftrightarrow
\dfrac{|\lambda_{i-1}|}{|\lambda_{i}|}=\dfrac{x_{p-1}+R}{x_{p-1}-R} \quad \forall \;i\in[2,p]
\\& \Leftrightarrow\dfrac{|\lambda_{i-1}|}{|\lambda_{i}|}=\dfrac{|\lambda_{p-1}|}{|\lambda_p|}  \quad \forall \;i\in[2,p].
\end{align*}
\end{proof}

\begin{proof}[Proof of Theorem \ref{MainThm2}]
The proof of the two claims follows from Lemmas \ref{LemmaCons} and \ref{lm:th2-2}.
\end{proof}

\section{Optimal result in the case of 3 negative eigenvalues}
\label{sec:3comp}

In this section, we assume that the matrix $A$ has three negative eigenvalues such that $\lambda_1<\lambda_2<\lambda_3<0$ ($n=p=3$). As in \eqref{FinalVn}, we have
 \begin{align}\label{FinalV3}
     \|V(t)\|_{L^2(\R)}\leq e^{-\gamma(t-\sup_{x\in\R}|\cI(x,t)|)}\,\|V_0\|_{L^2(\R)}.
 \end{align}
and the gross bound
\begin{equation}\label{grossbound}\sup_{x\geq R}|\cI|\leq \sum_{i=1}^3\dfrac{2R}{|\lambda_i|}=\bar{\tau},
\end{equation} 
which lead to a satisfying result concerning the large-time asymptotic behavior of the solution.
However, we now turn to a more detailed analysis of the solution for $t < \bar{\tau}$.

In the following, $\cD$ denotes the characteristic $X_3(\cdot,R,\dfrac{2R}{\lambda_3})$.
As in the proof of Lemma \ref{LemmaCons},  there exists a unique point $(x_2,t_2)\in \cD$ such that $ t_{2,en}(x_2,t_2)= t_{3,ex}(x_2,t_2)$ and we have
$$x_2=\dfrac{R(|\lambda_2|+|\lambda_3|)}{|\lambda_2|-|\lambda_3|} \quad \text{and} \quad t_2=2R\dfrac{|\lambda_2||\lambda_3|}{|\lambda_2|-|\lambda_3|}.$$

If we continue the procedure from Lemma \ref{LemmaCons} to find a point $(x_1,t_1)\in \cD$ such that $$t_{2,ex}(x_1,t_1)=t_{1,en}(x_1,t_1),$$ then such a point would not necessarily coincide with $(x_2,t_2)$ and thus verify 
$$t_{3,ex}(x_1,t_1)\ne t_{2,en}(x_1,t_1).$$

For $(x_2,t_2)$, the characteristics $X_2$ and $X_3$ are continuously inside $\omega^c$ for the time-length $\dfrac{2R}{|\lambda_3|}+\dfrac{2R}{|\lambda_2|}$. Concerning $X_1$, we have the two following scenarios:
\begin{enumerate}
 \item There is an \textit{overlap} between $X_2$ and $X_1$ inside $\omega^c$ -- i.e. $t_{2,ex}(x_2,t_2)-t_{1,en}(x_2,t_2)\leq 0$ -- and thus
 \begin{align*}\tau(x_2,t_2)&\leq \dfrac{2R}{|\lambda_3|}+\dfrac{2R}{|\lambda_2|}+\dfrac{2R}{|\lambda_1|}-\max(0,t_{2,ex}(x_2,t_2)-t_{1,en}(x_2,t_2))
\\&\leq \dfrac{2R}{|\lambda_3|}+\dfrac{2R}{|\lambda_2|}+\dfrac{2R}{|\lambda_1|}-\max\left(0,2R\dfrac{|\lambda_2|^2-|\lambda_1||\lambda_3|}{|\lambda_2||\lambda_1|(|\lambda_2|-|\lambda_3|)}\right).
\end{align*}
In this case, the three characteristics are continuously in $\omega^c$ but they are only inside $\omega^c$ for the maximum time $\bar{\tau}$ if $$\dfrac{\lambda_1}{\lambda_2}=\dfrac{\lambda_2}{\lambda_3}.$$
    \item There is a \textit{gap} between $X_2$ and $X_1$ -- i.e. $t_{2,ex}(x_2,t_2)-t_{1,en}(x_2,t_2)>0$ -- and thus $$|\cI(x_2,t_2)|=\frac{2R}{|\lambda_3|}+\frac{2R}{|\lambda_2|}+\frac{2R}{|\lambda_1|}. $$
   In this case the maximum time-length possible spent in $\omega^c$ is attained but in the interval $[t_{2,ex}(x_2,t_2),t_{1,ex}(x_2,t_2)]$ there are no characteristics inside $\omega^c$ and thus dissipation occurs.
\end{enumerate}

Let us turn to a detailed analysis of case (1): $t_{2,ex}(x_2,t_2)-t_{1,en}(x_2,t_2)<0$ and considering  $(x_1,t_1)\in\cD$ as the point verifying $t_{1,en}(x_1,t_1)= t_{2,ex}(x_1,t_1)$. Then the following result holds. 

\begin{lemma}\label{x1t1}
The point $(x_1,t_1)=\left(\dfrac{R(|\lambda_1|+|\lambda_2|)}{|\lambda_1|-|\lambda_2|},\dfrac{2R|\lambda_1||\lambda_3|}{|\lambda_1|-|\lambda_2|}\right)$ is such that
$$t_1=\argmin_{t\in\R^+}\{\sup_{x\in \R}\cI(x,t)\}=\bar{\tau}.$$
\end{lemma}
From this, it is clear that $t_1>\bar{\tau}$ and therefore using the bound \eqref{grossbound} on the time-interval $[0,t_1]$ is not sharp.
To conclude, we must now look at what happens between the points $(x_2,t_2)$ and $(x_1,t_1)$ more precisely.
\begin{lemma}
Let $(x,t)\in ((x_2,t_2),(x_1,t_1))$.
\begin{itemize}
    \item On the interval $[0,\tau_3]$, the characteristic $X_3(\cdot,x,t)$ is in $\omega^c$.
    \item On the interval $[\tau_3,t_{2,en}(x,t)]$ there are no characteristics crossing $(x,t)$ in $\omega^c$.
    % $\forall i=1,3\:X_i(.x,t)\notin \omega^c.$
    \item The characteristics $X_2(\cdot,x,t)$ and $X_1(\cdot,x,t)$ stay inside $\omega^c$ (they overlap) on the time-interval $[t_{2,en}(x_1,t_1),t_{1,ex}(x_1,t_1)]$ for a time-length of $t_{1,ex}(x_1,t_1)-t_{2,en}(x_1,t_1)$.
     \item On the interval $[t_{1,ex}(x_1,t_1),+\infty[$, there are not any characteristics passing through $(x,t)$ in $\omega^c$.
\end{itemize}
\end{lemma}
The following theorem is a direct consequence of our previous analysis and concludes the study of this particular case.
\begin{theorem} \label{Thm3comp}
Let $V$ be the solution of \eqref{eq:vn} with $n=p=3$ associated to the initial data $V_0\in L^1\cap L^2$.
\begin{itemize}

   \item For  $t\in[0,\tau_3+\tau_2+\tau_1-t_\lambda]$, $$\|V(t)\|_{L^2(\R)}\leq\|V_0\|_{L^2(\R)}\quad \text{where} \quad t_\lambda=\max\left(0,2R\dfrac{|\lambda_2|^2-|\lambda_1||\lambda_3|}{|\lambda_2||\lambda_1|(|\lambda_2|-|\lambda_3|)}\right).$$

   \item For $t\in[\tau_3+\tau_2+\tau_1-t_\lambda,t_{1,ex}(t)]$, $$\|V(t)\|_{L^2(\R)}\leq\|V_0\|_{L^2(\R)}e^{-\gamma(t_{2,en}(t)-\tau_3)}. $$
   \item For $t\in[t_{1,ex}(t),t_1]$ $$\|V(t)\|_{L^2(\R)}\leq\|V_0\|_{L^2(\R)}e^{-\gamma(t_1-t_{1,ex}(t))-\gamma(t_{2,en}(t)-\tau_3)}. $$
    \item And for $t\geq t_1$ $$\|V(t)\|_{L^2(\R)}\leq\|V_0\|_{L^2(\R)}e^{-\gamma(t-\tau)}. $$
\end{itemize}
\end{theorem}

\begin{remark}
Taking the curve made of the point from $$S=\left[\left(\|V_0\|_{L^2(\R)}e^{-\gamma(t_{2,en}(t)-\tau_3)},t_{1,ex}(t)\right), \quad \text{ for }(x,t)\in \mathcal{D}=[(x_2,t_2),(x_1,t_1)] \right]$$ leads the accurate convex envelop pictured in magenta in Figure \ref{3compdelay}.

Note that above the entering and exiting time only depend on $t$ as we are only interested in the points that are on $\cD$ and therefore satisfy $x=\lambda_3t-R$.
\end{remark}

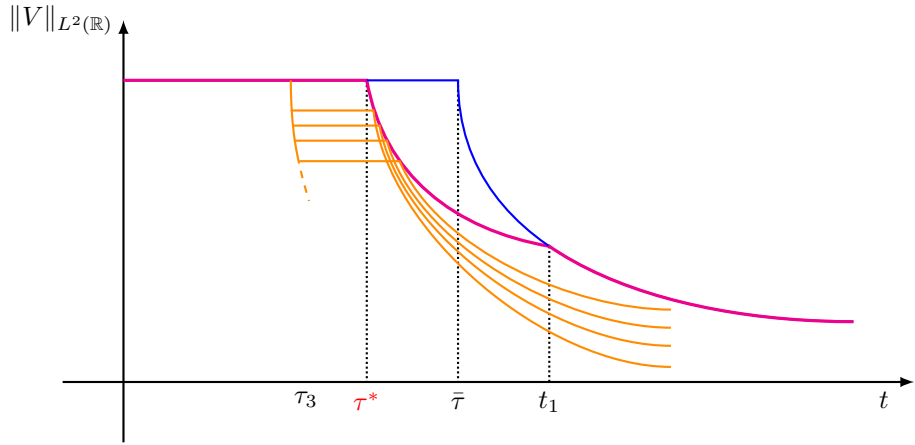
\begin{figure}[H]
\begin{tikzpicture}[xscale=0.8,yscale=0.8, thick]
\draw [-latex] (-1,0) -- (13,0) ;
\draw   (3,0) node [below] {$\tau_3$} (4,0) node [below, red] {$\tau^*$} (5.5,0) node [below] {$\bar{\tau}$} (7,0) node [below] {$t_1$} (12.5,0) node [below] {$t$} ; %t axis

\draw [-latex] (0,-1) -- (0,6) ;
\draw  (0,6)  node [left]  {$\|V\|_{L^2(\R)}$} ; %L2 norm

%\draw [cyan] (0,5) -- (5.5,5) to [out = 270, in = 134]  (7,2.25);
%\draw [magenta,very thick] (7,2.25) to [out = 280, in = 0]  (11,1);

\draw [densely dotted] (5.5,5) -- (5.5,0) ; 

%\draw [red]  (0,5) -- (4,5) to [out = 270, in = 180] (9,0);
\draw [densely dotted] (4,5) -- (4,0) ; 

\draw [magenta, very thick] (0,5) -- (4,5) to [out = 280, in = 170] (7,2.25);
\draw [densely dotted] (7,2.25) -- (7,0) ; 

\draw [blue, postaction={draw=magenta,very thick, dash pattern=on 0mm off 38mm on 80mm}] (4,5) -- (5.5,5) to [out = 270, in = 180]  (12,1);
\draw [orange, dashed, dash pattern=on 11.1mm off 1mm on 1mm off 1mm on 1mm off 1mm on 1mm] (2.75,5) to [bend right = 8] +(.3,-2);
\draw [orange] (2.75,4.5) -- (4.11,4.5) to [out = 276, in = 180, looseness = .8]  (9,.25);
\draw [orange] (2.77,4.25) -- (4.21,4.25) to [out = 280, in = 180, looseness = .8]  (9,.6);
\draw [orange] (2.8,4) -- (4.33,4) to [out = 285, in = 180, looseness = .8]  (9,.9);
\draw [orange] (2.89,3.66) -- (4.54,3.66) to [out = 293, in = 180, looseness = .8]  (9,1.2);

\end{tikzpicture}
\caption{Time-decay for 3 components.}  \label{3compdelay}
\end{figure}

\section{Extensions and open problems} 
\label{sec:open}

\subsection{More general undamped domain}
\label{ssec:general-omega}

Results similar to ours can be obtained whenever $\omega^c$ is a domain of finite measure. For example, let us consider $\omega^c$ as a finite union of bounded stripes; in this case, we can directly retrieve similar decay estimates with a delay depending on the time spent by each characteristics in each stripes. However, one must be careful when trying to recover an optimal result for small times in this case as there might be much more overlapping between each characteristics. For a sufficiently large time, we can recover the Shizuta-Kawashima decay rate with a delay equal to the sum of the time each characteristic spend in each stripes. We have the following theorem.
\begin{theorem}\label{ThmStripe}
Let $\omega^c=\bigcup_{j=1}^m[r_j,r_{j+1}]$. Assume that the matrix $A$ is symmetric, satisfies \eqref{ass:eigR} and that the couple $(A,B)$ satisfies the (SK) condition.
Let $U$ be the solution of \eqref{eq:LS} associated with the initial data $U_0\in L^1(\mathbb{R})\cap L^2(\mathbb{R})$.

There exists a finite time $\bar{\tau}>0$ such that for $t\geq\bar{\tau}$,
\begin{align*}
& \Vert U^h(\cdot,t) \Vert_{L^2(\R)} \le  e^{-\gamma (t-\bar{\tau})} \Vert U_0 \Vert_{L^2(\R)}, \\
&  \Vert U^\ell(\cdot,t) \Vert_{L^\infty(\R)} \le C (t-\bar{\tau})^{-1/2} \Vert U_0 \Vert_{L^1(\R)}
\end{align*}
\end{theorem}

In this setting, we are only able to prove the gross bound 
$$\displaystyle\bar{\tau}\leq \sum_{j=1}^m\sum_{i=1}^n\dfrac{r_{j+1}-r_j}{|\lambda_i|}.$$
By adapting the method developed here, one could also consider a domain made of periodically distributed damped and undamped stripes.

\subsection{Study on the half-line.}
With the method developed in the present paper, one can also consider the case when the x-space $\R$ is replaced by the half-line $(-\infty,0]$ and the undamped region is localized near the boundary. More precisely, consider the following linear hyperbolic system:
\begin{align}\label{eq:LSHalfLine}
\begin{cases}
\pt U + A \px U = B(x) U, & (x,t) \in (-\infty,0] \times (0,\infty), \\
U(0,x) = U_0(x), & x \in (-\infty,0],\\ \text{BC at } x=0.
\end{cases}
\end{align}
where $A$ and $B$ satisfy the same condition as before but $\omega$ is replaced by
\begin{align*}
\omega^* := \R \setminus [-R,0] = (-\infty,-R).
\end{align*}
In this case, if one supplement the system \eqref{eq:LSHalfLine} with suitable boundary conditions (BC) at $x=0$ such that the characteristics are reflected on the boundary, it is then ensured that the time spent by the characteristics in the undamped region $\omega^c$ is uniformly bounded and therefore the asymptotic result would follow from similar arguments as in our previous analysis.

Concerning the boundary conditions, we refer to \cite{Higdon, Russel}. In particular, to characterize the admissible boundary conditions, we should ensure that an uniform Kreiss condition is satisfied. We remark that, near the boundary $x=0$, we have $B(x)=0$; thus our system consists simply of uncoupled transport equations for which one can find suitable boundary conditions.

\subsection{Semilinear and quasilinear system}
\label{ssec:nonlinear}

In the semilinear case, if one assume that the time and space dependent eigenvalues conserve their ordering and signs for all $x$ and $t$, then the intuition is that similar results should hold true, as it is the case for the control of linear transport equations (or even some classes of nonlinear conservation laws, where a lower bound on the speed is needed; cf. \cite{MR3945166}). However, here one must be careful as the dependence on $x$ of the speed of each characteristics make the justification of the decay rates more difficult. More precisely, due to this dependence, one cannot reproduce exactly the computations done in the section dedicated to the scalar case $n=1$.
  
In the quasilinear case, the situation is even more delicate as one must first ensures that the solution exists locally in time and does not blow-up until the dynamics are damped. In the case where an analysis similar to that of the proof of Theorem \ref{MainThm2} can be realized in this case, it would lead to necessary conditions on the size of $\tau^*$ to guarantee the existence of the solution. We mention that such condition will depend only on the eigenvalues of the matrix $A$ and the size of the undamped domain $\omega^c$, and that it would not be trivial when dealing with many components, as seen in Section \ref{sec:3comp}.

\subsection{Higher dimensional setting}
\label{ssec:multid}

In the multidimensional setting, the general form of the system reads
\begin{equation}
\partial_tU + \sum_{j=1}^dA^j(U)\frac{\partial U}{\partial x_j}+B(x)V=0. \label{GEQSYM}
\end{equation}
In this context, there is a direct obstruction to the use of our arguments. Indeed, even in the linear case, the flux matrices $A^j$ may not all be diagonalizable in a same basis and therefore one is not able to rewrite the system as coupled transport equations. 

There are possible approaches, related to our method developed here:
\begin{itemize}
    \item Considering special solutions such as normal modes or plane waves (cf. \cite{Higdon}). These solutions propagate in a direction depending on the parameters given by the form of the solution. Under conditions on the direction of propagation depending on the form of the undamped domain, such a property would imply that the solution only stays in the undamped region for a finite time.
    \item More generally, one can consider the cone of propagation associated to hyperbolic system as defined in \cite{SerreGavage} for instance. Then, restricting the support of the initial data to a bounded domain and looking at the intersection of the cone of propagation and the undamped region, one could obtain a rough approximation of  the delay in the time-decay estimates.
\end{itemize}

\vspace{5mm}

\section*{Acknowledgments}
This work has received funding from the European Research Council (ERC) under the European
Union’s Horizon 2020 research and innovation programme (grant agreement NO: 694126-DyCon),
the Air Force Office of Scientific Research (AFOSR) under Award NO: FA9550-18-1-0242, the Grant
MTM2017-92996-C2-1-R COSNET of MINECO (Spain), the Alexander von Humboldt-Professorship
program, the European Unions Horizon 2020 research and innovation programme under the Marie Sklodowska-Curie grant agreement No.765579-ConFlex, and the Transregio 154 Project ``Mathematical Modelling, Simulation and Optimization Using the Example of Gas Networks'' of the Deutsche
Forschungsgemeinschaft.

Nicola De Nitti is a member of the Gruppo Nazionale per l’Analisi Matematica, la Probabilit\'a e le loro Applicazioni (GNAMPA) of the Istituto Nazionale di Alta Matematica (INdAM).

\vspace{5mm}
\bibliographystyle{abbrv}
\bibliography{PartiallyDissipativeLoc-v1-ref.bib}

\vfill

\end{document}